\documentstyle[amssymb,amsfonts]{amsart}

\def\normo#1{\left\|#1\right\|}
\def\normb#1{\big\|#1\big\|}
\def\abs#1{|#1|}
\def\aabs#1{\left|#1\right|}
\def\brk#1{\left(#1\right)}

\def\rev#1{\frac{1}{#1}}
\def\half#1{\frac{#1}{2}}
\def\norm#1{\|#1\|}
\def\jb#1{\langle#1\rangle}
\def\wt#1{\widetilde{#1}}

\newcommand{\N}{{\mathbb N}}
\newcommand{\T}{{\mathbb T}}

\newcommand{\R}{{\mathbb R}}
\newcommand{\C}{{\mathbb C}}
\newcommand{\Z}{{\mathbb Z}}
\newcommand{\ft}{{\mathcal{F}}}
\newcommand{\Hl}{{\mathcal{H}}}

\newcommand{\les}{{\lesssim}}
\newcommand{\ges}{{\gtrsim}}
\newcommand{\ra}{{\rightarrow}}

\newcommand{\Sch}{{\mathcal{S}}}

\newcommand{\sub}{{\mbox{sub}}}
\newcommand{\thd}{{\mbox{thd}}}
\numberwithin{equation}{section}

\theoremstyle{plain}
  \newtheorem{theorem}[subsection]{Theorem}
  \newtheorem{proposition}[subsection]{Proposition}
  \newtheorem{lemma}[subsection]{Lemma}
  \newtheorem{corollary}[subsection]{Corollary}
  
    \newtheorem*{conj}{Conjecture}

\theoremstyle{remark}
  \newtheorem{remark}[subsection]{Remark}

\theoremstyle{definition}
  \newtheorem{definition}[subsection]{Definition}

\newenvironment{proof}{\noindent {\bf Proof.} }{\endprf\par}
\def \endprf{\hfill  {\vrule height6pt width6pt depth0pt}\medskip}

\begin{document}
\title[Modified Benjamin-Ono equation]{Local Well-posedness and a priori bounds for the modified Benjamin-Ono equation without using a gauge
transformation}
\thanks{This work is supported in part by RFDP of China No.
20060001010, the National Science Foundation of China, grant
10571004; and the 973 Project Foundation of China, grant
2006CB805902, and the Innovation Group Foundation of NSFC, grant
10621061.}
\author{Zihua Guo}
\address{LMAM, School of Mathematical Sciences, Peking University, Beijing
100871, China}

\email{zihuaguo@@math.pku.edu.cn}

\urladdr{http://www.math.pku.edu.cn:8000/blog/gallery/167/}

\begin{abstract}
We prove that the complex-valued modified Benjamin-Ono (mBO)
equation is locally wellposed if the initial data $\phi$ belongs to
$H^s$ for $s\geq 1/2$ with $\norm{\phi}_{L^2}$ sufficiently small
without performing a gauge transformation. Hence the real-valued mBO
equation is globally wellposed for those initial datas, which is
contained in the results of C. Kenig and H. Takaoka \cite{KenigT}
where the smallness condition is not needed. We also prove that the
real-valued $H^\infty$ solutions to mBO equation satisfy a priori
local in time $H^s$ bounds in terms of the $H^s$ size of the initial
data for $s>1/4$.
\end{abstract}


\maketitle

\tableofcontents

\section{Introduction}
In this paper, we study the initial value problem for the
(defocusing) modified Benjamin-Ono equation of the form (also the
equation with focusing nonlinearity of the form $-u^2u_x$ can be
treated by our method)
\begin{eqnarray}\label{eq:mBO}
\begin{array}{l}
u_t+\Hl u_{xx}=u^2u_x,\ (x,t)\in \R^2,\\
u(x,0)=\phi(x),
\end{array}
\end{eqnarray}
where $u: \R^2 \ra \C$ is a complex-valued function and $\Hl$ is the
Hilbert transform which is defined as following
\begin{equation}\label{eq:Htran}
\Hl u(x)=\frac{1}{\pi} \mbox{p.v.}
\int_{-\infty}^{+\infty}\frac{u(y)}{x-y}dy.
\end{equation}
The equation with quadratic nonlinearity
\begin{equation}\label{eq:BO}
u_t+\Hl u_{xx}=uu_x
\end{equation}
was derived by Benjamin \cite{Ben} and Ono \cite{Ono} as a model for
one-dimensional waves in deep water. On the other hand, the cubic
nonlinearity is also of much interest for long wave models
\cite{ABFS,KPV}.

The initial value problems for \eqref{eq:mBO} and for the
Benjamin-Ono equation \eqref{eq:BO} have been extensively studied
\cite{BL, CKS, GV, GV2, GV3, HN, Iorio, KK, KPV2, KPV3, KenigT, KT,
KT2, MR, MR2, Ponce}. For instance, the energy method provides the
wellposedness on the Sobolev space $H^s$ for $s>3/2$ \cite{Iorio}.
For the Benjamin-Ono equation \eqref{eq:BO}, it has been know \cite
{KK, KT} that this is locally wellposed for $s>9/8$ by the
refinement of the energy method and dispersive estimates. Tao
\cite{Tao} obtained the global wellposedness in $H^s$ for $s\geq 1$
by performing a gauge transformation as for the derivative
Schr\"odinger equation and using the conservation law. This result
was improved by Ionescu and Kenig \cite{IK} who obtained global
wellposedness for $s\geq 0$, and also by Burq and Planchon \cite{BP}
who obtained local wellposedness for $s>1/4$.

For the modified Benjamin-Ono equation \eqref{eq:mBO}, Molinet and
Ribaud \cite{MR} obtained the local wellposedness in Sobolev space
$H^s$ for $s>1/2$. Their proof is based on Tao's gauge
transformation \cite{Tao}. The result for $s=1/2$ has been obtained
by Kenig and Takaoka \cite{KenigT} by using frequency dyadically
localized gauge transformation. Their result was sharp in the sense
that the solution map is not uniformly continuous in $H^s$ for
$s<1/2$. With the Sobolev space $H^s$ replaced by the Besov space
$B_{2,1}^s$, the result has been obtained in \cite{MR2} under a
smallness condition on the data. To the author's knowledge, these
results are all restricted to the real-valued mBO equation where the
gauge is easy to handle. Our method in this paper can also deal with
the complex-valued mBO equation.

The mBO equation \eqref{eq:mBO} has several symmetries. The first
one is the scaling invariance
\begin{equation}\label{eq:scaling}
u(x,t)\ra\
u_\lambda=\frac{1}{\lambda^{1/2}}u(\frac{x}{\lambda},\frac{t}{{\lambda^2}}),
\quad \phi_\lambda=\frac{1}{\lambda^{1/2}}\phi(\frac{x}{\lambda}),
\end{equation}
which leads to the constraint $s\geq 0$ on the wellposedness for
\eqref{eq:mBO}. We see that the equation \eqref{eq:mBO} is $L^2$
critical, hence the $L^2$ norm of the initial data is not
automatically small by the scaling, which is the main reason for us
assuming the initial data has a small $L^2$ norm. There are at least
the following three conservation laws preserved under the flow of
the real-valued mBO equation \eqref{eq:mBO}
\begin{eqnarray}\label{eq:conservation}
&&\frac{d}{dt}\int_\R u(x,t)dx=0, \\
&&\frac{d}{dt}\int_\R u(x,t)^2dx=0,\label{eq:L2con}\\
&&\frac{d}{dt}\int_\R \frac{1}{2}u\Hl u_x-\frac{1}{12}
|u(x,t)|^4dx=0.\label{eq:H1half}
\end{eqnarray}
These conservation laws provide a priori bounds on the solution. For
instance, we get from \eqref{eq:L2con} and \eqref{eq:H1half} that
the $H^{1/2}$ norm of the solution remains bounded for finite time
if the initial data belongs to $H^{1/2}$. Thus once one obtains a
solution of existence interval with a length determined by the
$H^{1/2}$ norm of the initial data, then the solution is
automatically extended to a global one.

In the first part of this paper, we reprove the results of Kenig and
Takaoka \cite{KenigT} without using a gauge transformation, but
under an extra condition that the $L^2$ norm of the initial data is
small. Since we don't perform a gauge transformation, our proof also
works for the complex-valued Cauchy problem \eqref{eq:mBO}. Our
method is to use the standard fixed-point argument in some Banach
space. Bourgain's space $X^{s,b}$ defined as a closure of the
following space
\[\{f\in \Sch(\R^2):\norm{f}_{X^{s,b}}=\norm{\jb{\xi}^s\jb{\tau-\omega(\xi)}^b\widehat{f}(\xi,\tau)}_{L^2}\}\]
is very useful in the study of the low regularity theory of the
nonlinear dispersive equations \cite{Bour, KPV4, IKT}. One might try
a direct perturbative approach in $X^{s,b}$ space as Kenig, Ponce
and Vega \cite{KPV4} did for the KdV and modified KdV equations.
However, one will find that the key trilinear estimate
\begin{equation}\label{eq:trilinearX}
\norm{\partial_x(u^3)}_{X^{s,b-1}}\les \norm{u}_{X^{s,b}}^3, \
\mbox{ for some } b\in [1/2,1)
\end{equation}
 fails for any $s$ due to logarithmic
divergences involving the modulation variable (see Proposition
\ref{countertrilinear}, \ref{counterXsb} below). The key observation
in this paper is that these logarithmic divergences can be removed
by us using Banach spaces which combine $X^{s,b}$ structure with
smoothing effect structure. The spaces of these structures were
first found and used by Ionescu and Kenig \cite{IK} to remove some
logarithmic divergence.

\begin{theorem}\label{t11}
Let $s\geq 1/2$. Assume $u_0 \in {H}^{s}$ with $\norm{u_0}_{L^2}\ll
1$. Then

(a) Existence. There exists $T=T(\norm{u_0}_{H^{1/2}})>0$ and a
solution $u$ to the complex-valued mBO equation \eqref{eq:mBO} (or
its focusing version) satisfying
\begin{equation}
u\in F^{s}(T)\subset C([-T,T]:H^{s}),
\end{equation}
where the function space $F^{s}(T)$ will be defined later (see
section 2).

(b) Uniqueness. For the real-valued case, the solution mapping
$u_0\rightarrow u$ is the unique extension of the mapping
$H^\infty\rightarrow C([-T,T]:H^\infty)$. For the complex-valued
case, the solution $u$ is unique in $B(u_0)$ which is defined in
\eqref{eq:Bu0}.

(c) Lipschitz continuity. For any $R>0$, the mapping $u_0\rightarrow
u$ is Lipschitz continuous from $\{u_0\in
H^{s}:\norm{u_0}_{H^{s}}<R, \norm{u_0}_{L^2}\ll 1 \}$ to
$C([-T,T]:H^{s})$.
\end{theorem}

For the real-valued mBO equation, from the conservation laws
\eqref{eq:L2con}, \eqref{eq:H1half}, and iterating Theorem
\ref{t11}, we obtain the following corollary.

\begin{corollary}
The Cauchy problem for the real-valued mBO equation \eqref{eq:mBO}
(or its focusing version) is globally wellposed if $\phi$ belongs to
$H^s$ for $s\geq 1/2$ with $\norm{\phi}_{L^2}$ sufficiently small
\end{corollary}

In the second part, we study the low regularity problem of the
real-valued mBO equation \eqref{eq:mBO}. From the ill-posedness
result in \cite{KenigT}, we see that for $s<1/2$ one can't use a
direct contraction mapping method, but we expect some wellposedness
results hold in the weak sense.

\begin{conj} The solution map $S_T^\infty: H^\infty \rightarrow C([-T,T]:H^\infty)$ of the real-valued modified Benjamin-Ono equation
\eqref{eq:mBO} can be uniquely extended to a continuous map from
$H^s$ to $C([-T,T]:H^s)$ for a small $T=T(\norm{\phi}_{H^s})>0$ if
$s>1/4$.
\end{conj}

To prove this one would need to establish a priori $H^s$ bounds for
the $H^\infty$ solutions and then prove continuous dependence on the
initial data. We solve the easier half of this problem.
\begin{theorem}\label{aprioribound}
Let $s>1/4$. For any $M>0$ there exists $T>0$ and $C>0$ so that for
any initial data $u_0\in H^{\infty}$ satisfying
\begin{eqnarray*}
\norm{u_0}_{H^s}\leq M
\end{eqnarray*}
then the solutions $u\in C([0,T]:H^{\infty})$ to the real-valued mBO
equation \eqref{eq:mBO} satisfies
\begin{equation}
\norm{u}_{H^s}\leq C \norm{u_0}_{H^s}.
\end{equation}
\end{theorem}

\begin{remark}
A very similar equation to \eqref{eq:mBO} is the derivative
nonlinear Schr\"odinger equation
\begin{eqnarray}\label{eq:dNLS}
u_t-iu_{xx}=|u|^2u_x, \quad (x,t)\in \R^2,
\end{eqnarray}
and local wellposedness was known for the equation in $H^s$ for
$s\geq 1/2$ \cite{Takaoka}, where a fixed point argument is
performed in an adapted Bourgain's $X^{s,b}$ space after a gauge
transformation on the equation. Our methods also give the same
results in Theorem \ref{t11} and \ref{aprioribound} for
\eqref{eq:dNLS}.
\end{remark}

We discuss now some of the ingredients in the proof of Theorem
\ref{aprioribound}. We will follow the method of Ionescu, Kenig and
Tataru \cite{IKT} which approaches the problem in a less
perturbative way. It can be viewed as a combination of the energy
method and the perturbative method. More precisely, we will define
$F^{l,s}(T), N^{l,s}(T)$ and energy space $E^{l,s}(T)$ and show that
if $u$ is a smooth solution of \eqref{eq:mBO} on $\R\times [-T,T]$
with $\norm{u}_{E^{l,s}(T)}\ll 1$, then
\begin{eqnarray}\label{eq:scheme}
\left \{
\begin{array}{l}
\norm{u}_{F^{l,s}(T)}\les \norm{u}_{E^{l,s}(T)}+\norm{\partial_x(u^3)}_{N^{l,s}(T)};\\
\norm{\partial_x(u^3)}_{N^{l,s}(T)}\les \norm{u}_{F^{l,s}(T)}^3;\\
\norm{u}_{E^{l,s}(T)}^2\les \norm{\phi}_{\dot{H}^l\cap
\dot{H}^s}^2+\norm{u}_{F^{l,s}(T)}^3.
\end{array}
\right.
\end{eqnarray}
The inequalities \eqref{eq:scheme} and a simple continuity argument
still suffice to control $\norm{u}_{F^{l,s}(T)}$, provided that
$\norm{\phi}_{\dot{H}^l\cap \dot{H}^s}\ll 1$ (which can by arranged
by rescaling if $l,s>0$). The first inequality in \eqref{eq:scheme}
is the analogue of the linear estimate. The second inequality in
\eqref{eq:scheme} is the analogue of the trilinear estimate
\eqref{eq:trilinearX}. The last inequality in \eqref{eq:scheme} is
an energy-type estimate.

We explain the strategies in \cite{IKT} to define the main normed
and semi-normed spaces. As was explained before, standard using of
$X^{s,b}$ spaces for fixed-point argument will lead to a logarithmic
divergence in the key trilinear estimate. But we use $X^{s,b}$-type
structures only on small, frequency dependant time intervals. The
high-low interaction can be controlled for short time. The second
step is to define $\norm{u}_{E^{l,s}(T)}$ sufficiently large to be
able to still prove the linear estimate $\norm{u}_{F^{l,s}(T)}\les
\norm{u}_{E^{l,s}(T)}+\norm{\partial_x(u^3)}_{N^{l,s}(T)}$. Finally,
we use frequence-localized energy estimates and the symmetries of
the equation \eqref{eq:mBO} to prove the energy estimate
$\norm{u}_{E^{l,s}(T)}^2\les \norm{\phi}_{\dot{H}^l\cap
\dot{H}^s}^2+\norm{u}_{F^{l,s}(T)}^3$.

The rest of the paper is organized as follows: In section 2 we
present some notations and Banach function spaces. We summarize some
properties of the spaces in section 3. A symmetric estimate will be
given in section 4 which is used in section 5 to show dyadic
trilinear estimate. The proof of Theorem \ref{t11} is given in
section 6. In section 7 we prove short time dyadic trilinear
estimate and in section 8 we prove an energy estimate. Finally in
section 9 we prove Theorem \ref{aprioribound}.

\section{Notation and Definitions}
For $x, y\in \R^+$, $x\les y$ means that there exists $C>0$ such
that $x\leq Cy$. By $x\sim y$ we mean $x\les y$ and $y\les x$. For
$f\in \Sch'$ we denote by $\widehat{f}$ or $\ft (f)$ the Fourier
transform of $f$ for both spatial and time variables,
\begin{eqnarray*}
\widehat{f}(\xi, \tau)=\int_{\R^2}e^{-ix \xi}e^{-it \tau}f(x,t)dxdt.
\end{eqnarray*}
Besides, we use $\ft_x$ and $\ft_t$ to denote the Fourier transform
with respect to space and time variable respectively. Let
$\Z_+=\Z\cap[0, \infty)$. For $k\in \Z$ let
\[O_k=\{\xi:|\xi|\in [(3/4)\cdot 2^k,
(3/2)\cdot 2^k)\ \}\mbox{ and }{I}_k=\{\xi: |\xi|\in [2^{k-1},
2^{k+1}]\ \}.\] For $k\in \Z_+$ let $\widetilde{I}_k=[-2,2]$ if
$k=0$ and $\widetilde{I}_k=I_k$ if $k\geq 1$.

 Let $\eta_0: \R\rightarrow [0, 1]$ denote an even
smooth function supported in $[-8/5, 8/5]$ and equal to $1$ in
$[-5/4, 5/4]$. For $k\in \Z$ let
$\chi_k(\xi)=\eta_0(\xi/2^k)-\eta_0(\xi/2^{k-1})$, $\chi_k$
supported in $\{\xi: |\xi|\in[(5/8)\cdot 2^k, (8/5)\cdot 2^k]\}$,
and
\[\chi_{[k_1,k_2]}=\sum_{k=k_1}^{k_2}\chi_k \mbox{ for any } k_1\leq k_2\in \Z.\]
For simplicity of notation, let $\eta_k=\chi_k$ if $k\geq 1$ and
$\eta_k\equiv 0$ if $k\leq -1$. Also, for $k_1\leq k_2\in \Z$ let
\[\eta_{[k_1,k_2]}=\sum_{k=k_1}^{k_2}\eta_k \mbox{ and }\eta_{\leq k_2}=\sum_{k=-\infty}^{k_2}\eta_{k}.\]
Roughly speaking, $\{\chi_k\}_{k\in \mathbb{Z}}$ is the homogeneous
decomposition function sequence and $\{\eta_k\}_{k\in \mathbb{Z}_+}$
is the non-homogeneous decomposition function sequence to the
frequency space.

For $k\in \Z$ let $k_+=\max(k,0)$, and let $P_k$, $R_k$ denote the
operators on $L^2(\R)$ defined by
\[
\widehat{P_ku}(\xi)=\chi_k(\xi)\widehat{u}(\xi) \mbox{ and }
\widehat{R_ku}(\xi)=1_{O_k}(\xi)\widehat{u}(\xi).
\]
By a slight abuse of notation we also define the operators $P_k$,
$R_k$ on $L^2(\R\times \R)$ by formulas $\ft(P_ku)(\xi,
\tau)=\chi_k(\xi)\ft(u)(\xi, \tau)$ and $\ft(R_ku)(\xi,
\tau)=1_{O_k}(\xi)\ft(u)(\xi, \tau)$. For $l\in \Z$ let
\[
P_{\leq l}=\sum_{k\leq l}P_k, \quad P_{\geq l}=\sum_{k\geq l}P_k.
\] Similarly we also define the operators $R_{\leq l}$ and $R_{\geq l}$.

Let $a_1, a_2, a_3, a_4\in \R$. It will be convenient to define the
quantities $a_{max}\geq a_{sub}\geq a_{thd}\geq a_{min}$ to be the
maximum, sub-maximum, third-maximum, and minimum of $a_1,a_2,a_3,
a_4$ respectively. We also denote $\sub(a_1,a_2,a_3,a_4)=a_{sub}$
and $\thd(a_1,a_2,a_3,a_4)=a_{thd}$. Usually we use
$k_1,k_2,k_3,k_4$ and $j_1,j_2,j_3,j_4$ to denote integers,
$N_i=2^{k_i}$ and $L_i=2^{j_i}$ for $i=1,2,3,4$ to denote dyadic
numbers. For $a\in \R$ we define $a-$ to be the real number
$a-\epsilon$ for some $0<\epsilon\ll 1$. Similar we also define
$a+$.

For $\xi\in \R$ let
\begin{equation}\label{eq:dr}
\omega(\xi)=-\xi|\xi|,
\end{equation}
which is the dispersion relation associated to the linear
Benjamin-Ono equation. For $\phi \in L^2(\R)$ let $W(t)\phi\in
C(\R:L^2)$ denote the solution of the free Benjamin-Ono evolution
given by
\begin{equation}
\ft_x[W(t)\phi](\xi,t)=e^{it\omega(\xi)}\widehat{\phi}(\xi),
\end{equation}
where $\omega(\xi)$ is defined in \eqref{eq:dr}. For $k\in \Z_+$ and
$j\geq 0$ let $D_{k,j}=\{(\xi, \tau)\in \R \times \R: \xi \in
\widetilde{I}_k, \tau-\omega(\xi)\in \widetilde{I}_j\}$. For $k\in
\Z$ and $j\geq 0$ let $\dot{D}_{k,j}=\{(\xi, \tau)\in \R \times \R:
\xi \in {I}_k, \tau-\omega(\xi)\in \widetilde{I}_j\}$. For $k\in
\Z_+$ we define now the Banach spaces $X_k(\R\times \R)$:
\begin{eqnarray}\label{eq:Xk}
X_k&=&\{f\in L^2(\R^2): f \mbox{ is supported in } \widetilde
{I}_k\times\R \mbox{ and }\nonumber\\
&&\norm{f}_{X_k}:=\sum_{j=0}^\infty
2^{j/2}\beta_{k,j}\norm{\eta_j(\tau-w(\xi))\cdot
f(\xi,\tau)}_{L^2_{\xi,\tau}}<\infty\},
\end{eqnarray}
where
\begin{equation}\label{eq:betakj}
\beta_{k,j}=1+2^{(j-2k)/2}.
\end{equation}
Here the spaces $X_k$ is the same as those used by Ionescu and Kenig
\cite{IK} for $k>0$. $X_0$ is different, since we don't have the
special structures in the low frequency. The precise choice of the
coefficients $\beta_{k,j}$ is important in order for all the
trilinear estimates in Section 5 to hold. We see that if $k$ is
small then $\beta_{k,j}\approx 2^{j/2}$. This factor is particularly
important in controlling the high-low interaction. From the
technical level, we know from the K-Z method of Tao \cite{Taokz}
that the worst interaction is that the low frequency component has a
largest modulation. But the factor $\beta_{k,j}$ will compensate for
that. The logarithmic divergence caused by the other interaction,
namely high frequency component with largest modulation, can be
removed by using a smoothing effect structure.

As in \cite{IK}, the spaces $X_k$ are not sufficient for our
purpose, due to various logarithmic divergences involving the
modulation variable (See Proposition \ref{countertrilinear} below).
For $k\geq 100$ we also define the Banach spaces $Y_k=Y_k(\R^2)$.
For $k\geq 100$ we define
\begin{eqnarray}\label{eq:Yk}
Y_k&=&\{f\in L^2(\R^2): f \mbox{ is supported in }
\bigcup_{j=0}^{k-1}D_{k,j} \mbox{ and }\nonumber\\&&
\norm{f}_{Y_k}:=2^{-k/2}\norm{\ft^{-1}[(\tau-\omega(\xi)+i)f(\xi,\tau)]}_{L_x^1L_t^2}<\infty\}.
\end{eqnarray}
Then for $k\in \Z_+$ we define
\begin{equation}\label{eq:Zk}
Z_k:=X_k \mbox{ if } k\leq 99 \mbox{ and } Z_k:=X_k+Y_k \mbox{ if }
k\geq 100.
\end{equation}
The spaces $Z_k$ are our basic Banach spaces. For $s\geq 0$ we
define the Banach spaces $F^{s}=F^{s}(\R\times\R)$:
\begin{eqnarray}\label{eq:Fs}
F^{s}=\{u\in \Sch'(\R\times \R):
\norm{u}_{F^{s}}^2=\sum_{k=0}^{\infty}2^{2sk}\norm{\eta_k(\xi)\ft
(u)}_{Z_k}^2<\infty\},
\end{eqnarray}
and $N^{s}=N^{s}(\R\times\R)$ which is used to measure the nonlinear
term and can be viewed as an analogue of $X^{s,b-1}$
\begin{eqnarray}\label{eq:Ns}
N^{s}&=&\{u\in \Sch'(\R\times \R):\nonumber\\
&&\norm{u}_{N^{s}}^2=\sum_{k=0}^{\infty}2^{2sk}\norm{\eta_k(\xi)(\tau-\omega(\xi)+i)^{-1}\ft
(u)}_{Z_k}^2<\infty\}.
\end{eqnarray}
The spaces $F^s$ and $N^s$ have the same structures in high
frequency as those in \cite{IK}, but with different structures in
low frequency.

In order to prove a priori bounds, we will need another set of norms
and semi-norms which were first used by Ionescu, Kenig and Tataru in
\cite{IKT} for the KP-I equation. Similar idea can be found in
\cite{KochTataru}. For $k\in \Z$ we define
\begin{eqnarray}\label{eq:Bk}
B_k&=&\{f\in L^2(\R^2): f \mbox{ is supported in } I_k\times\R
\mbox{ and }\nonumber\\&& \norm{f}_{B_k}:=\sum_{j=0}^\infty
2^{j/2}\norm{\eta_j(\tau-w(\xi))\cdot
f(\xi,\tau)}_{L^2_{\xi,\tau}}<\infty\}.
\end{eqnarray}
These $l^1$-type $X^{s,b}$ structures were first introduced and used
in \cite{Tataru, IK, IKT}. It is also useful in the study of uniform
global wellposedness and inviscid limit for the KdV-Burgers equation
\cite{GW}.

At frequency $2^k$ we will use the $X^{s, b}$ structure given by the
$B_k$ norm, uniformly on the $2^{-k_+}$ time scale. For $k\in \Z$ we
define the normed spaces
\begin{eqnarray*}
&& F_k=\left\{
\begin{array}{l}
f\in L^2(\R^2): \widehat{f} \mbox{ is supported in }
I_k\times\R \mbox{ and }\\
\norm{f}_{F_k}=\sup_{t_k\in \R}\norm{\ft[f\cdot
\eta_0(2^{k_+}(t-t_k))]}_{B_k}<\infty
\end{array}
\right\},
\\
&&N_k=\left\{
\begin{array}{l}
f\in L^2(\R^2): \widehat{f} \mbox{ is supported in }
I_k\times\R \mbox{ and } \norm{f}_{N_k}=\\
\sup_{t_k\in \R}\norm{(\tau-\omega(\xi)+i2^{k_+})^{-1}\ft[f\cdot
\eta_0(2^{k_+}(t-t_k))]}_{B_k}<\infty
\end{array}
\right\}.
\end{eqnarray*}
The bounds we obtain for smooth solutions of the equation
\eqref{eq:mBO} are on a fixed time interval, while the above
function spaces are not. Thus we define a local version of the
spaces. For $T\in (0,1]$ we define the normed spaces
\begin{eqnarray*}
F_k(T)&=&\{f\in C([-T,T]:E_k): \norm{f}_{F_k(T)}=\inf_{\wt{f}=f
\mbox{ in } \R\times [-T,T]}\norm{\wt f}_{F_k}\};\\
N_k(T)&=&\{f\in C([-T,T]:E_k): \norm{f}_{N_k(T)}=\inf_{\wt{f}=f
\mbox{ in } \R\times [-T,T]}\norm{\wt f}_{N_k}\}.
\end{eqnarray*}
For $l,s\geq 0$ and $T\in (0,1]$, we define the normed spaces
\begin{eqnarray*}
&&F^{l,s}(T)=\left\{
\begin{array}{l}
u\in C([-T,T]:H^\infty):\quad \norm{u}_{F^{l,s}}^2=\\
\sum_{k=-\infty}^{0}2^{2lk}\norm{R_k(u)}_{F_k(T)}^2+\sum_{k=1}^{\infty}2^{2sk}\norm{R_k(u)}_{F_k(T)}^2<\infty
\end{array}
\right\},
\\
&&N^{l,s}(T)=\left\{
\begin{array}{l}
u\in
C([-T,T]:H^\infty):\quad \norm{u}_{N^{l,s}}^2=\\
\sum_{k=-\infty}^{0}2^{2lk}\norm{R_k(u)}_{N_k(T)}^2+\sum_{k=1}^{\infty}2^{2sk}\norm{R_k(u)}_{N_k(T)}^2<\infty
\end{array}
\right\}.
\end{eqnarray*}
We still need an energy space. For $l,s\geq 0$ and $u\in
C([-T,T]:H^\infty)$ we define
\begin{eqnarray*}
\norm{u}_{E^{l,s}(T)}^2=\norm{R_{\leq
0}(u(0))}_{\dot{H}^l}^2+\sum_{k\geq 1}\sup_{t_k\in
[-T,T]}2^{2sk}\norm{R_k(u(t_k))}_{L^2}^2.
\end{eqnarray*}

\section{Properties of the Spaces $Z_k$}

In this section we devote to study the properties of the spaces
$Z_k$. Using the definitions, if $k\in \Z_+$ and $f_k \in Z_k$ then
$f_k$ can be written in the form
\begin{eqnarray}
\left \{
\begin{array}{l}
f_k=\sum_{j=0}^{\infty}f_{k,j}+g_k;\\
\sum_{j=0}^{\infty}2^{j/2}\beta_{k,j}\norm{f_{k,j}}_{L^2}+\norm{g_k}_{Y_k}\leq
2\norm{f_k}_{Z_k},
\end{array}
\right.
\end{eqnarray}
such that $f_{k,j}$ is supported in $D_{k,j}$ and $g_k$ is supported
in $\cup_{j=0}^{k-1}D_{k,j}$ (if $k\leq 99$ then $g_k\equiv 0$). We
start with the elementary properties.
\begin{lemma}[Lemma 4.1, \cite{IK}]\label{l31} (a) If $m, m':\R\rightarrow \C$,
$k\in \Z_+$, and $f_k\in Z_k$ then
\begin{eqnarray}
\left \{
\begin{array}{l}
\norm{m(\xi)f_k(\xi,\tau)}_{Z_k}\leq C\norm{\ft^{-1}(m)}_{L^1(\R)}\norm{f_k}_{Z_k};\\
\norm{m'(\tau)f_k(\xi,\tau)}_{Z_k}\leq
C\norm{m}_{L^\infty(\R)}\norm{f_k}_{Z_k}.
\end{array}
\right.
\end{eqnarray}

(b) If $k\in \Z_+$, $j\geq 0$, and $f_k\in Z_k$ then
\begin{equation}
\norm{\eta_j(\tau-\omega(\xi))f_k(\xi,\tau)}_{X_k}\leq
C\norm{f_k}_{Z_k}.
\end{equation}

(c) If $k\geq 1$, $j\in [0,k]$, and $f_k$ is supported in $I_k\times
\R$ then
\begin{equation}
\norm{\ft^{-1}[\eta_{\leq
j}(\tau-\omega(\xi))f_k(\xi,\tau)]}_{L_x^1L_t^2}\leq
C\norm{\ft^{-1}(f_k)}_{L_x^1L_t^2}.
\end{equation}

\end{lemma}

We study now the embedding properties of the spaces $Z_k$. We will
see that the spaces $X_k$ and $Y_k$ are very close.
\begin{lemma}\label{l32}
Let $k\in \Z_+$, $s\in \R$ and $I\subset \R$ be an interval. Let $Y$
be $L_x^pL_{t\in I}^q$ or $L_{t\in I}^qL_x^p$ for some $1\leq
p,q\leq \infty$ with the property that
\[\norm{e^{-t\Hl\partial_x^2}f}_Y\les 2^{ks}\norm{f}_{L^2(\R)}\]
for all $f\in L^2(\R)$ with $\widehat{f}$ supported in $\wt{I}_k$
and $\tau_0\in \R$. Then we have that if $f_k\in Z_k$
\begin{equation}
\norm{\ft^{-1}(f_k)}_{Y}\les 2^{ks}\norm{f_k}_{Z_k}.
\end{equation}
\end{lemma}
\begin{proof}
We assume first that $f_k=f_{k,j}$ with
$\norm{f_k}_{X_k}=2^{j/2}\beta_{k,j}\norm{f_{k,j}}_{L^2}$ and
$f_{k,j}$ is supported in $D_{k,j}$ for some $j\geq 0$. Then we have
\begin{eqnarray*}
\ft^{-1}(f_{k})(x,t)&=&\int f_{k,j}(\xi,\tau) e^{ix\xi}e^{it\tau}d\xi d\tau\\
&=&\int_{\widetilde{I}_j} e^{it\tau} \int
f_{k,j}(\xi,\tau+\omega(\xi)) e^{ix\xi}e^{it\omega(\xi)}d\xi d\tau.
\end{eqnarray*}
From the hypothesis on $Y$, we obtain
\begin{eqnarray*}
\norm{\ft^{-1}(f_{k})(x,t)}_Y &\les&  \int \eta_j(\tau)
\normo{e^{it\tau} \int f_{k,j}(\xi,\tau+\omega(\xi))
e^{ix\xi}e^{it\omega(\xi)}d\xi}_{Y}d\tau\\
&\les& 2^{ks}2^{j/2}\norm{f_{k,j}}_{L^2},
\end{eqnarray*}
which completes the proof in this case.

We assume now that $k\geq 100$ and $f_k=g_k\in Y_k$. From definition
$g_k$ can be written in the form
\begin{eqnarray}\label{eq:gkform}
\left \{
\begin{array}{l}
g_k(\xi,\tau)=2^{k/2}\chi_{[k-1,k+1]}(\xi)(\tau-\omega(\xi)+i)^{-1}\eta_{\leq k}(\tau-\omega(\xi))\ft_x h(\xi,\tau);\\
\norm{g_k}_{Y_k}=C\norm{h}_{L_x^1L_\tau^2}.
\end{array}
\right.
\end{eqnarray}
It suffices to prove that if
\[
f(\xi,\tau)=2^{k/2}\chi_{[k-1,k+1]}(\xi)(\tau-\omega(\xi)+i)^{-1}\eta_{\leq
k}(\tau-\omega(\xi))\cdot h(\tau)
\]
then
\begin{equation}
\normo{\int_{\R^2}f(\xi,\tau)e^{ix\xi}e^{it\tau}d\xi d\tau}_{Y}\les
2^{ks}\norm{h}_{L^2},
\end{equation}
which follows from the proof of Lemma 4.2 (b) in \cite{IK}.
\end{proof}

In order to obtain the more specific embedding properties of the
spaces $Z_k$, we need the estimate for the free Benjamin-Ono
equation. We recall the Strichartz estimates, smoothing effects, and
maximal function estimates (for the proof, see, e.g. \cite{KeelT,
KPV5} and the reference therein)

\begin{lemma}\label{l33}
Let $k\in \Z_+$ and $I\subset \R$ be an interval with $|I|\les 1$.
Then for all $\phi \in L^2(\R)$ with $\widehat{\phi}$ supported in
$\wt{I}_k$,

(a) Strichartz estimates
\begin{eqnarray*}
\norm{e^{-t\Hl\partial_x^2}\phi}_{L_t^qL_x^r}\leq
C\norm{\phi}_{L^2(\R)},
\end{eqnarray*}
where $(q,r)$ is admissible, namely $2\leq q,r\leq \infty$ and
$2/q=1/2-1/r$.

(b) Smoothing effect
\begin{eqnarray*}
\norm{e^{-t\Hl\partial_x^2}\phi}_{L_x^\infty L_t^2}\leq
C2^{-k/2}\norm{\phi}_{L^2(\R)}.
\end{eqnarray*}

(c) Maximal function estimate
\begin{eqnarray*}
\norm{e^{-t\Hl\partial_x^2}\phi}_{L_x^2 L_{t\in I}^\infty}&\leq&
C2^{k/2}\norm{\phi}_{L^2(\R)},\\
\norm{e^{-t\Hl\partial_x^2}\phi}_{L_x^4 L_{t}^\infty}&\leq&
C2^{k/4}\norm{\phi}_{L^2(\R)}.
\end{eqnarray*}
\end{lemma}

In particular we note the case $(6,6)$ is admissible which we will
use in the sequel. From Lemma \ref{l32}, \ref{l33}, we immediately
get the following

\begin{lemma}\label{l34}
Let $k \in \Z_+$ and $I\subset \R$ be an interval with $|I|\les 1$.
Assume $(q,r)$ is admissible and $f_k \in Z_k$. Then
\begin{eqnarray*}
\begin{array}{l}
\norm{\ft^{-1}(f_k)}_{L_x^2L_{t\in I}^\infty}\leq C
2^{k/2}\norm{f_k}_{Z_k}.\\
\norm{\ft^{-1}(f_k)}_{L_x^\infty L_t^2}\leq C
2^{-k/2}\norm{f_k}_{Z_k},\\
\norm{\ft^{-1}(f_k)}_{L_x^4 L_t^\infty}\leq C
2^{k/4}\norm{f_k}_{Z_k},\\
\norm{\ft^{-1}(f_k)}_{L_t^q L_x^r}\leq C \norm{f_k}_{Z_k}.
\end{array}
\end{eqnarray*}
As a consequence,
\begin{equation}
F^{s}\subseteq C(\R; {H}^s) \mbox{ for any } s\geq 0.
\end{equation}
\end{lemma}

Now we turn to study the properties of the space $F^{l,s}$. The
definition shows easily that if $k\in \Z$ and $f_k\in B_k$ then
\begin{eqnarray}\label{eq:pBk1}
\normo{\int_{\R}|f_k(\xi,\tau')|d\tau'}_{L_\xi^2}\les
\norm{f_k}_{B_k}.
\end{eqnarray}
Moreover, if $k\in \Z$, $l\in \Z_+$, and $f_k\in B_k$ then
\begin{eqnarray}\label{eq:pBk2}
&&\sum_{j=l+1}^\infty 2^{j/2}\normo{\eta_j(\tau-\omega(\xi))\cdot
\int_{\R}|f_k(\xi,\tau')|\cdot
2^{-l}(1+2^{-l}|\tau-\tau'|)^{-4}d\tau'}_{L^2}\nonumber\\
&&+2^{l/2}\normo{\eta_{\leq l}(\tau-\omega(\xi))\cdot \int_{\R}
|f_k(\xi,\tau')|\cdot
2^{-l}(1+2^{-l}|\tau-\tau'|)^{-4}d\tau'}_{L^2}\nonumber\\
&&\les \norm{f_k}_{B_k}.
\end{eqnarray}
In particular, if $k\in \Z$, $l\in \Z_+$, $t_0\in \R$, $f_k \in
B_k$, and $\gamma\in \Sch(\R)$, then
\begin{eqnarray}\label{eq:pBk3}
\norm{\ft[\gamma(2^l(t-t_0))\cdot \ft^{-1}(f_k)]}_{B_k}\les
\norm{f_k}_{B_k}.
\end{eqnarray}

Indeed, to prove \eqref{eq:pBk2}, first for the second term on the
left-hand side of \eqref{eq:pBk2}, we immediately get from
Cauchy-Schwarz inequality and \eqref{eq:pBk1} that
\begin{eqnarray*}
&&2^{l/2}\normo{\eta_{\leq l}(\tau-\omega(\xi))\cdot \int_{\R}
|f_k(\xi,\tau')|\cdot
2^{-l}(1+2^{-l}|\tau-\tau'|)^{-4}d\tau'}_{L^2}\nonumber\\
&&\les \normo{\int_{\R} |f_k(\xi,\tau')|d\tau'}_{L_\xi^2}\les
\norm{f_k}_{B_k}.
\end{eqnarray*}
For the first term on the left-hand side of \eqref{eq:pBk2}, we
decompose $f_k(\xi,\tau')=\sum_{j_1\geq 0}f_{k,j_1}$ where
$f_{k,j_1}=f_k(\xi,\tau')\eta_{j_1}(\tau'-\omega(\xi))$, and then we
get
\begin{eqnarray*}
&&\sum_{j=l+1}^\infty 2^{j/2}\normo{\eta_j(\tau-\omega(\xi))\cdot
\int_{\R}|f_k(\xi,\tau')|\cdot
2^{-l}(1+2^{-l}|\tau-\tau'|)^{-4}d\tau'}_{L^2}\nonumber\\
&&\les \sum_{j=l+1}^\infty \sum_{j_1=0}^\infty
2^{j/2}\normo{\eta_j(\tau-\omega(\xi))
\int_{\R}|f_{k,j_1}(\xi,\tau')|\cdot
2^{-l}(1+2^{-l}|\tau-\tau'|)^{-4}d\tau'}_{L^2}\nonumber\\
&&=\sum_{j=l+1}^\infty (\sum_{j_1>j+5}+\sum_{j_1<j-5}+\sum_{j-5\leq
j_1\leq j+5})\nonumber\\
&&\quad 2^{j/2}\normo{\eta_j(\tau-\omega(\xi))
\int_{\R}|f_{k,j_1}(\xi,\tau')|\cdot
2^{-l}(1+2^{-l}|\tau-\tau'|)^{-4}d\tau'}_{L^2}\\
&&=:I+II+III.
\end{eqnarray*}
For the contribution of $I$, we first observe that $|\tau-\tau'|\sim
2^{j_1}$ in this case. Then we get that
\begin{eqnarray*}
I\les \sum_{j_1\geq l}\sum_{j\leq j_1}2^j
2^{3l}2^{-4j_1}\normo{\int_{\R}|f_{k,j_1}(\xi,\tau')|d\tau'}_{L_\xi^2}\leq
\norm{f_k}_{B_k}.
\end{eqnarray*}
Similarly we can estimate the contribution of $II$. For the third
term $III$, using Young's inequality, then we get
\begin{eqnarray*}
III\les \sum_{j=l+1}^\infty \sum_{|j-j_1|\leq 5}2^{j/2}
\norm{f_{k,j_1}}_{L^2}\les \norm{f_k}_{B_k}.
\end{eqnarray*}

As in \cite{IKT}, for any $k\in \Z$ we define the set $S_k$ of
$k-acceptable$ time multiplication factors
\begin{eqnarray}
S_k=\{m_k:\R\rightarrow \R: \norm{m_k}_{S_k}=\sum_{j=0}^{10}
2^{-jk_+}\norm{\partial^jm_k}_{L^\infty}< \infty\}.
\end{eqnarray}
For instance, $\eta(2^{k_+}t) \in S_k$ for any $\eta$ satisfies
$\norm{\partial_x^k \eta}_{L^\infty}\leq C$ for $j=1,2,\ldots, 10$.
Direct estimates using the definitions and \eqref{eq:pBk2} show that
for any $s\geq 0$ and $T\in (0,1]$
\begin{eqnarray}\label{eq:Sk}
\left \{
\begin{array}{l}
\norm{\sum_{k\in \Z} m_k(t)\cdot R_k(u)}_{F^{l,s}(T)}\les (\sup_{k\in \Z}\norm{m_k}_{S_k})\cdot \norm{u}_{F^{l,s}(T)};\\
\norm{\sum_{k\in \Z} m_k(t)\cdot R_k(u)}_{N^{l,s}(T)}\les
(\sup_{k\in \Z}\norm{m_k}_{S_k})\cdot \norm{u}_{N^{l,s}(T)};\\
\norm{\sum_{k\in \Z} m_k(t)\cdot R_k(u)}_{E^{l,s}(T)}\les
(\sup_{k\in \Z}\norm{m_k}_{S_k})\cdot \norm{u}_{E^{l,s}(T)}.
\end{array}
\right.
\end{eqnarray}

Actually, for instance we show the first inequality in
\eqref{eq:Sk}. In view of definition, it suffices to prove that if
$u_k \in F_k$, then
\begin{eqnarray*}
\norm{m_k(t)u_k}_{F_k}\les \norm{u_k}_{F_k}\norm{m_k}_{S_k}, \quad
\forall \ k\in \Z.
\end{eqnarray*}
From \eqref{eq:pBk3} we see that we only need to prove that
\begin{eqnarray*}
\big|\ft\big[m_k(\cdot)\eta_0(2^{k_+}(\cdot-t_k))\big]\big|\les
2^{-k_+} (1+2^{-k_+}|\tau|)^{-4}\norm{m_k}_{S_k},
\end{eqnarray*}
which follows from partial integration and the definition of $S_k$.

\section{A Symmetric Estimate}

In this section we prove a symmetric estimate which will be used to
prove a trilinear estimate. For $\xi_1,\xi_2,\xi_3\in \R$ and
$\omega:\R \rightarrow \R$ as in \eqref{eq:dr} let
\begin{equation}\label{eq:reso}
\Omega(\xi_1,\xi_2,\xi_3)=\omega(\xi_1)+\omega(\xi_2)+\omega(\xi_3)-\omega(\xi_1+\xi_2+\xi_3).
\end{equation}
This is the resonance function that plays a crucial role in the
trilinear estimate of the $X^{s,b}$-type space. See \cite{Taokz} for
a perspective discussion. For compactly supported functions $f,g,h,u
\in L^2(\R\times \R)$ let
\begin{eqnarray*}
&&J(f,g,h,u)=\int_{\R^6}f(\xi_1,\mu_1)g(\xi_2,\mu_2)h(\xi_3,\mu_3)\\
&&\qquad
u(\xi_1+\xi_2+\xi_3,\mu_1+\mu_2+\mu_3+\Omega(\xi_1,\xi_2,\xi_3))d\xi_1d\xi_2d\xi_3d\mu_1d\mu_2d\mu_3.
\end{eqnarray*}

\begin{lemma}\label{l41}
Assume $k_1,k_2,k_3,k_4 \in \Z$ and $k_1\leq k_2\leq k_3\leq k_4$,
$j_1,j_2,j_3,j_4\in \Z_+$, and $f_{k_i,j_i}\in L^2(\R\times \R)$ are
nonnegative functions supported in $I_{k_i}\times
\cup_{l=0}^{j_i}\widetilde{I}_{l}, \ i=1,\ 2,\ 3,\ 4$. For
simplicity we write
$J=|J(f_{k_1,j_1},f_{k_2,j_2},f_{k_3,j_3},f_{k_4,j_4})|$. Then

(a) For any $k_1\leq k_2\leq k_3\leq k_4$ and $j_1,j_2,j_3,j_4\in
\Z_+$,
\begin{equation}
J\leq C 2^{(j_{min}+j_{thd})/2}2^{(k_{min}+k_{thd})/2}
\prod_{i=1}^4\norm{f_{k_i,j_i}}_{L^2}.
\end{equation}

(b) If $k_2\leq k_3-5$ and $j_2\neq j_{max}$,
\begin{equation}\label{eq:l41rb}
J\leq C
2^{(j_1+j_2+j_3+j_4)/2}2^{-j_{max}/2}2^{-k_{max}/2}2^{k_{min}/2}
\prod_{i=1}^4\norm{f_{k_i,j_i}}_{L^2};
\end{equation}
if $k_2\leq k_3-5$ and $j_2=j_{max}$,
\begin{equation}
J\leq C
2^{(j_1+j_2+j_3+j_4)/2}2^{-j_{max}/2}2^{-k_{max}/2}2^{k_{thd}/2}
\prod_{i=1}^4\norm{f_{k_i,j_i}}_{L^2}.
\end{equation}

(c) For any $k_1,k_2,k_3,k_4 \in \Z$ and $j_1,j_2,j_3,j_4\in \Z_+$,
\begin{equation}
J\leq C2^{(j_1+j_2+j_3+j_4)/2}2^{-j_{max}/2}
\prod_{i=1}^4\norm{f_{k_i,j_i}}_{L^2}.
\end{equation}

(d) If $k_{min}\leq k_{max}-10$, then
\begin{equation}\label{eq:l41d}
J\leq C2^{(j_1+j_2+j_3+j_4)/2}2^{-k_{max}}
\prod_{i=1}^4\norm{f_{k_i,j_i}}_{L^2}.
\end{equation}

\end{lemma}
\begin{proof}
Let $A_{k_i}(\xi)=[\int_\R |f_{k_i,j_i}(\xi,\mu)|^2d\mu]^{1/2}$,
$i=1,2,3,4$. Using the Cauchy-Schwarz inequality and the support
properties of the functions $f_{k_i,j_i}$,
\begin{eqnarray*}
&&|J(f_{k_1,j_1},f_{k_2,j_2},f_{k_3,j_3},f_{k_4,j_4})|\\
&\leq& C
2^{(j_{min}+j_{thd})/2}\int_{\R^3}A_{k_1}(\xi_1)A_{k_2}(\xi_1)A_{k_3}(\xi_1)A_{k_4}(\xi_1+\xi_2+\xi_3)d\xi_1d\xi_2d\xi_3\\
&\leq& C
2^{(k_{min}+k_{thd})/2}2^{(j_{min}+j_{thd})/2}\prod_{i=1}^4\norm{f_{k_i,j_i}}_{L^2},
\end{eqnarray*}
which is part (a), as desired.

For part (b), in view of the support properties of the functions, it
is easy to see that
$J(f_{k_1,j_1},f_{k_2,j_2},f_{k_3,j_3},f_{k_4,j_4})\equiv 0$ unless
\begin{equation}\label{eq:freeq}
k_4\leq k_3+5.
\end{equation}
Simple changes of variables in the integration and the observation
that the function $\omega$ is odd show that
\begin{eqnarray*}
|J(f,g,h,u)|=|J(g,f,h,u)|=|J(f,h,g,u)|=|J(\widetilde{f},g,u,h)|,
\end{eqnarray*}
where $\widetilde{f}(\xi,\mu)=f(-\xi,-\mu)$. We assume first that
$j_2\neq j_{max}$. Then we have several cases: if $j_4=j_{max}$,
then we will prove that if $g_i:\R\rightarrow \R_+$ are $L^2$
functions supported in $I_{k_i}$, $i=1,2,3$, and $g: \R^2\rightarrow
\R_+$ is an $L^2$ function supported in $I_{k_4}\times
\widetilde{I}_{j_4}$, then
\begin{eqnarray}\label{eq:l41b}
&&\int_{\R^3}g_1(\xi_1)g_2(\xi_2)g_3(\xi_3)g(\xi_1+\xi_2+\xi_3,\Omega(\xi_1,\xi_2,\xi_3))d\xi_1d\xi_2d\xi_3\nonumber\\
&&\les
2^{-k_{max}/2}2^{k_{min}/2}\norm{g_1}_{L^2}\norm{g_2}_{L^2}\norm{g_3}_{L^2}\norm{g}_{L^2}.
\end{eqnarray}
This suffices for \eqref{eq:l41rb}.

To prove \eqref{eq:l41b}, we first observe that since $k_2\leq
k_3-5$ then $|\xi_3+\xi_2|\sim |\xi_3|$. By change of variable
$\xi'_1=\xi_1$, $\xi'_2=\xi_2$, $\xi'_3=\xi_3+\xi_2$, we get that
the left side of \eqref{eq:l41b} is bounded by
\begin{eqnarray}\label{eq:l41b2}
&&\int_{|\xi_1|\sim 2^{k_1},|\xi_2|\sim 2^{k_2},|\xi_3|\sim
2^{k_3}}g_1(\xi_1)g_2(\xi_2)\nonumber\\
&&g_3(\xi_3-\xi_2)g(\xi_1+\xi_3,\Omega(\xi_1,\xi_2,\xi_3-\xi_2))d\xi_1d\xi_2d\xi_3.
\end{eqnarray}
Note that in the integration area we have
\begin{eqnarray*}
\big|\frac{\partial}{\partial_{\xi_2}}\left[\Omega(\xi_1,\xi_2,\xi_3-\xi_2)\right]\big|=|\omega'(\xi_2)-\omega'(\xi_3-\xi_2)|\sim
2^{k_3},
\end{eqnarray*}
where we use the fact $\omega'(\xi)=|\xi|$ and $k_2\leq k_3-5$. By
change of variable $\mu_2=\Omega(\xi_1,\xi_2,\xi_3-\xi_2)$, we get
that \eqref{eq:l41b2} is bounded by
\begin{eqnarray}
&&2^{-k_3/2}\int_{|\xi_1|\sim
2^{k_1}}g_1(\xi_1)\norm{g_2}_{L^2}\norm{g_3}_{L^2}\norm{g}_{L^2}d\xi_1 \nonumber\\
&\les&2^{-k_{max}/2}2^{k_{min}/2}\norm{g_1}_{L^2}\norm{g_2}_{L^2}\norm{g_3}_{L^2}\norm{g_1}_{L^2}\norm{g}_{L^2}.
\end{eqnarray}

If $j_3=j_{max}$, this case is identical to the case $j_4=j_{max}$
in view of \eqref{eq:freeq}. If $j_1=j_{max}$ it suffices to prove
that if $g_i:\R\rightarrow \R_+$ are $L^2$ functions supported in
$I_{k_i}$, $i=2,3,4$, and $g: \R^2\rightarrow \R_+$ is an $L^2$
function supported in $I_{k_1}\times \widetilde{I}_{j_1}$, then
\begin{eqnarray}\label{eq:l41b3}
&&\int_{\R^3}g_2(\xi_2)g_3(\xi_3)g_4(\xi_4)g(\xi_2+\xi_3+\xi_4,\Omega(\xi_2,\xi_3,\xi_4))d\xi_2d\xi_3d\xi_4\nonumber\\
&&\les
2^{-k_{max}/2}2^{k_{min}/2}\norm{g_2}_{L^2}\norm{g_3}_{L^2}\norm{g_4}_{L^2}\norm{g}_{L^2}.
\end{eqnarray}

Indeed, by change of variables
$\xi'_2=\xi_2,\xi'_3=\xi_3,\xi'_4=\xi_2+\xi_3+\xi_4$ and noting that
in the area $|\xi'_2|\sim 2^{k_2},|\xi'_3|\sim 2^{k_3},|\xi'_4|\sim
2^{k_1}$,
\begin{eqnarray*}
\big|\frac{\partial}{\partial_{\xi'_2}}\left[\Omega(\xi'_2,\xi'_3,\xi'_4-\xi'_2-\xi'_3)\right]\big|=|\omega'(\xi'_2)-\omega'(\xi'_4-\xi'_2-\xi'_3)|\sim
2^{k_3},
\end{eqnarray*}
we get from Cauchy-Schwarz inequality that
\begin{eqnarray}
&&\int_{\R^3}g_2(\xi_2)g_3(\xi_3)g_4(\xi_4)g(\xi_2+\xi_3+\xi_4,\Omega(\xi_2,\xi_3,\xi_4))d\xi_2d\xi_3d\xi_4\nonumber\\
&\les& \int_{|\xi'_2|\sim 2^{k_2},|\xi'_3|\sim 2^{k_3}, |\xi'_4|\sim
2^{k_1}}g_2(\xi'_2)g_3(\xi'_3)\nonumber\\
&&\quad \cdot
g_4(\xi'_4-\xi'_2-\xi'_3)g(\xi'_4,\Omega(\xi'_2,\xi'_3,\xi'_4-\xi'_2-\xi'_3))d\xi'_2d\xi'_3d\xi'_4\nonumber\\
&\les&2^{-k_3/2}\int_{|\xi'_3|\sim 2^{k_3}, |\xi'_4|\sim
2^{k_1}}g_3(\xi'_3)\norm{g_2(\xi'_2)g_4(\xi'_4-\xi'_2-\xi'_3)}_{L_{\xi'_2}^2}\norm{g(\xi'_4,\cdot)}_{L_{\xi'_2}^2}d\xi'_3d\xi'_4\nonumber\\
&\les&2^{-k_{max}/2}2^{k_{min}/2}\norm{g_2}_{L^2}\norm{g_3}_{L^2}\norm{g_4}_{L^2}\norm{g}_{L^2}.
\end{eqnarray}

We assume now that $j_2=j_{max}$. This case is identical to the case
$j_1=j_{max}$. We note that we actually prove that if $k_2\leq
k_3-5$ then
\begin{eqnarray}\label{eq:l41bb}
J\leq C
2^{(j_1+j_2+j_3+j_4)/2}2^{-j_{sub}/2}2^{-k_{max}/2}2^{k_{min}/2}\prod_{i=1}^4\norm{f_{k_i,j_i}}_{L^2}.
\end{eqnarray}
Therefore, we complete the proof for part (b).

For part (c), setting $f^{\sharp}_{k_i,j_i}=f_{k_i,j_i}(\xi,
\tau-\omega(\xi))$, $i=1,2,3$, then we get
\begin{eqnarray*}
|J(f_{k_1,j_1},f_{k_2,j_2},f_{k_3,j_3},f_{k_4,j_4})|&=&|\int
f^\sharp_{k_1,j_1}*f^\sharp_{k_2,j_2}*f^\sharp_{k_3,j_3}\cdot
f^\sharp_{k_4,j_4}|\\
&\les&
\norm{f^\sharp_{k_1,j_1}*f^\sharp_{k_2,j_2}*f^\sharp_{k_3,j_3}}_{L^2}\norm{f^\sharp_{k_4,j_4}|}_{L^2}\\
&\les&\prod_{i=1}^3\norm{\ft^{-1}(f^\sharp_{k_i,j_i})}_{L^6}\norm{f_{k_4,j_4}|}_{L^2}.
\end{eqnarray*}
On the other hand, from
\begin{eqnarray*}
\ft^{-1}(f^\sharp_{k_1,j_1})&=&\int_{\R^2}f_{k_i,j_i}(\xi,
\tau-\omega(\xi))e^{ix\xi}e^{it\tau}d\xi d\tau\\
&=&\int_{\R^2} f_{k_i,j_i}(\xi,
\tau)e^{ix\xi}e^{it\omega(\xi)}e^{it\tau}d\xi d\tau,
\end{eqnarray*}
then it follows from Lemma \ref{l33} (a) that
\[\norm{\ft^{-1}(f^\sharp_{k_1,j_1})}_{L^6}\les  \int_{\R}\normo{\int_\R f_{k_i,j_i}(\xi,
\tau)e^{ix\xi}e^{it\omega(\xi)}d\xi}_{L^6} d\tau\les
2^{j_i/2}\norm{f_{k_i,j_i}}_{L^2}.\] Thus part (c) follows form the
symmetry.

For part (d), we only need to consider the worst cases $\xi_1\cdot
\xi_2<0$ and $k_2\leq k_3-5$. Indeed in the other cases we get
\eqref{eq:l41d} from the fact $|\Omega(\xi_1,\xi_2,\xi_3)|\ges
2^{k_2+k_3}$ which implies that $j_{max}\geq k_2+k_3-20$ by checking
the support properties. Thus (d) follows from (b) and (c) in these
cases. We assume now $\xi_1\cdot \xi_2<0$ and $k_2\leq k_3-5$. If
$j_4=j_{max}$, it suffices to prove that if $g_i$ is $L^2$
nonnegative functions supported in $I_{k_i}$, $i=1,2,3$, and $g$ is
a $L^2$ nonnegative function supported in $I_{k_4}\times
\widetilde{I}_{j_4}$, then
\begin{eqnarray}\label{eq:l41d1}
&&\int_{\R^3\cap \{\xi_1\cdot
\xi_2<0\}}g_1(\xi_1)g_2(\xi_2)g_3(\xi_3)g(\xi_1+\xi_2+\xi_3,
\Omega(\xi_1,\xi_2,\xi_3))d\xi_1d\xi_2d\xi_3\nonumber\\
&\les&
2^{j_4/2}2^{-k_3}\norm{g_1}_{L^2}\norm{g_2}_{L^2}\norm{g_3}_{L^2}\norm{g}_{L^2}.
\end{eqnarray}
By localizing $|\xi_1+\xi_2|\sim 2^l$ for $l\in \Z$, we get that the
right-hand side of \eqref{eq:l41d1} is bounded by
\begin{eqnarray}\label{eq:l41d2}
\sum_{l}\int_{\R^3}\chi_{l}(\xi_1+\xi_2)g_1(\xi_1)g_2(\xi_2)g_3(\xi_3)g(\xi_1+\xi_2+\xi_3,
\Omega(\xi_1,\xi_2,\xi_3))d\xi_1d\xi_2d\xi_3.
\end{eqnarray}
From the support properties of the functions $g_i,\ g$ and the fact
that in the integration area
\[|\Omega(\xi_1,\xi_2,\xi_3)|=(\xi_1+\xi_2)(\xi_1+\xi_3)\sim 2^{l+k_3},\]
We get that
\begin{equation}\label{eq:l41dsi}
j_{max}\geq l+k_3-20.
\end{equation}
By changing variable of integration $\xi_1'=\xi_1+\xi_2$,
$\xi_2'=\xi_2$, $\xi_3'=\xi_1+\xi_3$, we obtain that
\eqref{eq:l41d2} is bounded by
\begin{eqnarray}\label{eq:l41d3}
&&\sum_{l}\int_{|\xi_1'|\sim 2^l,|\xi_2'|\sim 2^{k_2},|\xi_3'|\sim 2^{k_3}}\chi_{l}(\xi_1')g_1(\xi_1'-\xi_2')g_2(\xi_2')g_3(\xi_2'+\xi_3'-\xi_1')\nonumber\\
&&\quad g(\xi_2'+\xi_3',
\Omega(\xi_1'-\xi_2',\xi_2',\xi_2'+\xi_3'-\xi_1'))d\xi_1'd\xi_2'd\xi_3'.
\end{eqnarray}
Since in the integration area
\begin{eqnarray}\label{eq:l41djo}
\big|\frac{\partial}{\partial_{\xi_1'}}[\Omega(\xi_1'-\xi_2',\xi_2',\xi_2'+\xi_3'-\xi_1')]\big|
=|\omega'(\xi_1'-\xi_2')-\omega'(\xi_2'+\xi_3'-\xi_1')|\sim 2^{k_3},
\end{eqnarray}
then we get from \eqref{eq:l41djo} that \eqref{eq:l41d3} is bounded
by
\begin{eqnarray}
&&\sum_{l}\int_{|\xi_1'|\sim 2^l}\chi_{l}(\xi_1')\norm{g_1}_{L^2}\norm{g_3}_{L^2}\nonumber\\
&&\quad \norm{g_2(\xi_2')g(\xi_2'+\xi_3',
\Omega(\xi_1'-\xi_2',\xi_2',\xi_2'+\xi_3'-\xi_1'))}_{L^2_{\xi_2',\xi_3'}}d\xi_1'\nonumber\\
&\les&\sum_{l}2^{l/2}2^{-k_3/2}\norm{g_1}_{L^2}\norm{g_2}_{L^2}\norm{g_3}_{L^2}\norm{g}_{L^2}\nonumber\\
&\les&
2^{j_{max}/2}2^{-k_3}\norm{g_1}_{L^2}\norm{g_2}_{L^2}\norm{g_3}_{L^2}\norm{g}_{L^2},
\end{eqnarray}
where we used \eqref{eq:l41dsi} in the last inequality.

From symmetry we know the case $j_3=j_{max}$ is identical to the
case $j_4=j_{max}$, and the case $j_1=j_{max}$ is identical to the
case $j_2=j_{max}$, thus it reduces to prove the case $j_2=j_{max}$.
It suffices to prove that if $g_i$ is $L^2$ nonnegative functions
supported in $I_{k_i}$, $i=1,3,4$, and $g$ is a $L^2$ nonnegative
function supported in $I_{k_2}\times \widetilde{I}_{j_2}$, then
\begin{eqnarray}\label{eq:l41dc21}
&&\int_{\R^3\cap \{\xi_1\cdot
\xi_2<0\}}g_1(\xi_1)g_3(\xi_3)g_4(\xi_4)g(\xi_1+\xi_3+\xi_4,
\Omega(\xi_1,\xi_3,\xi_4))d\xi_1d\xi_3d\xi_4\nonumber\\
&\les&
2^{j_2/2}2^{-k_3}\norm{g_1}_{L^2}\norm{g_4}_{L^2}\norm{g_3}_{L^2}\norm{g}_{L^2}.
\end{eqnarray}
As the case $j_4=j_{max}$, we get that the right-hand side of
\eqref{eq:l41dc21} is bounded by
\begin{eqnarray}\label{eq:l41dc22}
\sum_{l}\int_{\R^3}\chi_{l}(\xi_3+\xi_4)g_1(\xi_1)g_4(\xi_4)g_3(\xi_3)g(\xi_1+\xi_4+\xi_3,
\Omega(\xi_1,\xi_3,\xi_4))d\xi_1d\xi_4d\xi_3.
\end{eqnarray}
From the support properties of the functions $g_i,\ g$ and the fact
that in the integration area
\[|\Omega(\xi_1,\xi_2,\xi_3)|=|(\xi_1+\xi_4)(\xi_4+\xi_3)|\sim 2^{l+k_3},\]
We get that
\begin{equation}\label{eq:l41dc2si}
j_{max}\geq l+k_3-20.
\end{equation}
By changing variable of integration $\xi_1'=\xi_1+\xi_3$,
$\xi_3'=\xi_3+\xi_4$, $\xi_4'=\xi_1+\xi_3+\xi_4$, we obtain that
\eqref{eq:l41dc22} is bounded by
\begin{eqnarray}\label{eq:l41dc23}
&&\sum_{l}\int_{|\xi_3'|\sim 2^l,|\xi_4'|\sim 2^{k_2},|\xi_1'|\sim 2^{k_3}}\chi_{l}(\xi_3')g_1(\xi_4'-\xi_3')g_3(\xi_1'+\xi_3'-\xi_4')g_4(\xi_4'-\xi_1')\nonumber\\
&&\quad g(\xi_4',
\Omega(\xi_4'-\xi_3',\xi_1'+\xi_3'-\xi_4',\xi_4'-\xi_1'))d\xi_1'd\xi_3'd\xi_4'.
\end{eqnarray}
Since in the integration area,
\begin{eqnarray}\label{eq:l41dc2jo}
&&\big|\frac{\partial}{\partial_{\xi_3'}}[\Omega(\xi_4'-\xi_3',\xi_1'+\xi_3'-\xi_4',\xi_4'-\xi_1')]\big|\nonumber\\
&=&|-\omega'(\xi_4'-\xi_3')+\omega'(\xi_1'+\xi_3'-\xi_4')|\sim
2^{k_3},
\end{eqnarray}
then we get from \eqref{eq:l41dc2jo} that \eqref{eq:l41dc23} is
bounded by
\begin{eqnarray}
&&\sum_{l}\int_{|\xi_3'|\sim 2^l}\chi_{l}(\xi_3')\norm{g_1}_{L^2}\norm{g_3}_{L^2}\nonumber\\
&&\quad \norm{g_4(\xi_4'-\xi_1')g(\xi_4',
\Omega(\xi_4'-\xi_3',\xi_1'+\xi_3'-\xi_4',\xi_4'-\xi_1'))}_{L^2_{\xi_1',\xi_4'}}d\xi_3'\nonumber\\
&\les&\sum_{l}2^{l/2}2^{-k_3/2}\norm{g_1}_{L^2}\norm{g_3}_{L^2}\norm{g_4}_{L^2}\norm{g}_{L^2}\nonumber\\
&\les&
2^{j_{max}/2}2^{-k_3}\norm{g_1}_{L^2}\norm{g_2}_{L^2}\norm{g_3}_{L^2}\norm{g}_{L^2},
\end{eqnarray}
where we used \eqref{eq:l41dc2si} in the last inequality. Therefore,
we complete the proof for part (d).
\end{proof}

We restate now Lemma \ref{l41} in a form that is suitable for the
trilinear estimates in the next sections.
\begin{corollary}\label{cor42}
Assume $k_1,k_2,k_3,k_4\in \Z$, $j_1,j_2,j_3,j_4\in \Z_+$, and
$f_{k_i,j_i}\in L^2(\R\times \R)$ are functions supported in
$\dot{D}_{k_i,j_i}$, $i=1,2$.

(a) For any $k_1,k_2,k_3,k_4\in \Z$ and $j_1,j_2,j_3,j_4\in \Z_+$,
\begin{eqnarray}
&&\norm{1_{\dot{D}_{k_4,j_4}}(\xi,\tau)(f_{k_1,j_1}*f_{k_2,j_2}*f_{k_3,j_3})}_{L^2}\nonumber\\
&\leq&
C2^{(k_{min}+k_{thd})/2}2^{(j_{min}+j_{thd})/2}\prod_{i=1}^3\norm{f_{k_i,j_i}}_{L^2}.
\end{eqnarray}

(b) For any $k_1,k_2,k_3,k_4\in \Z$ with $k_{thd}\leq k_{max}-5$,
and $j_1,j_2,j_3,j_4\in \Z_+$. If for some $i\in \{1,2,3,4\}$ such
that $(k_i,j_i)=(k_{thd},j_{max})$, then
\begin{eqnarray}
&&\norm{1_{\dot{D}_{k_4,j_4}}(\xi,\tau)(f_{k_1,j_1}*f_{k_2,j_2}*f_{k_3,j_3})}_{L^2}\nonumber\\
&\leq&
C2^{(-k_{max}+k_{thd})/2}2^{(j_{1}+j_{2}+j_3+j_4)/2}2^{-j_{max}/2}\prod_{i=1}^3\norm{f_{k_i,j_i}}_{L^2},
\end{eqnarray}
else we have
\begin{eqnarray}
&&\norm{1_{\dot{D}_{k_4,j_4}}(\xi,\tau)(f_{k_1,j_1}*f_{k_2,j_2}*f_{k_3,j_3})}_{L^2}\nonumber\\
&\leq&
C2^{(-k_{max}+k_{min})/2}2^{(j_{1}+j_{2}+j_3+j_4)/2}2^{-j_{max}/2}\prod_{i=1}^3\norm{f_{k_i,j_i}}_{L^2}.
\end{eqnarray}

(c) For any $k_1,k_2,k_3,k_4\in \Z$ and $j_1,j_2,j_3,j_4\in \Z_+$,
\begin{eqnarray}
&&\norm{1_{\dot{D}_{k_4,j_4}}(\xi,\tau)(f_{k_1,j_1}*f_{k_2,j_2}*f_{k_3,j_3})}_{L^2}\nonumber\\
&\leq&
C2^{(j_{1}+j_{2}+j_3+j_4)/2}2^{-j_{max}/2}\prod_{i=1}^3\norm{f_{k_i,j_i}}_{L^2}.
\end{eqnarray}

(d) For any $k_1,k_2,k_3,k_4\in \Z$ with $k_{min}\leq k_{max}-10$,
and $j_1,j_2,j_3,j_4\in \Z_+$,
\begin{eqnarray}
&&\norm{1_{\dot{D}_{k_4,j_4}}(\xi,\tau)(f_{k_1,j_1}*f_{k_2,j_2}*f_{k_3,j_3})}_{L^2}\nonumber\\
&\leq&
C2^{(j_{1}+j_{2}+j_3+j_4)/2}2^{-k_{max}/2}\prod_{i=1}^3\norm{f_{k_i,j_i}}_{L^2}.
\end{eqnarray}
\end{corollary}
\begin{proof}
Clearly, we have
\begin{eqnarray}
&&\norm{1_{\dot{D}_{k_4,j_4}}(\xi,\tau)(f_{k_1,j_1}*f_{k_2,j_2}*f_{k_3,j_3})(\xi,\tau)}_{L^2}\nonumber\\
&=&\sup_{\norm{f}_{L^2}=1}\aabs{\int_{D_{k_4,j_4}} f\cdot
f_{k_1,j_1}*f_{k_2,j_2}*f_{k_3,j_3} d\xi d\tau}.
\end{eqnarray}
Let $f_{k_3,j_3}=1_{D_{k_4,j_4}}\cdot f$, and then
$f_{k_i,j_i}^\sharp(\xi,\mu)=f_{k_i,j_i}(\xi,\mu+\omega(\xi))$,
$i=1,2,3,4$. The functions $f_{k_i,j_i}^\sharp$ are supported in
$I_{k_i}\times \cup_{|m|\leq 3}\wt{I}_{j_i+m}$,
$\norm{f_{k_i,j_i}^\sharp}_{L^2}=\norm{f_{k_i,j_i}}_{L^2}$. Using
simple changes of variables, we get
\[\int_{D_{k_4,j_4}} f\cdot
f_{k_1,j_1}*f_{k_2,j_2}*f_{k_3,j_3} d\xi d\tau =
J(f_{k_1,j_1}^\sharp,f_{k_2,j_2}^\sharp,f_{k_3,j_3}^\sharp,f_{k_4,j_4}^\sharp).\]
Then Corollary \ref{cor42} follows from Lemma \ref{l41}.
\end{proof}

\begin{remark}\label{rem43}
From the proof, we see that Lemma \ref{l41} and Corollary
\ref{cor42} also hold if $k_1,k_2,k_3,k_4 \in \Z_+$ and with
$I_{k_i}$ replaced by $\wt{I}_{k_i}$, $\dot{D}_{k_i,j_i}$ replaced
by $\cup_{l=0}^{j_i}D_{k_i,l}$. The methods can be used to deal with
a general dispersion relation $\omega(\xi)$.
\end{remark}

\section{Trilinear Estimates}

In this section we devote to prove some dyadic trilinear estimates,
using the symmetric estimates obtained in the last section. We
divide it into several cases. The first case is $low\times high
\rightarrow high$ interactions.

\begin{proposition}\label{p51} Assume $k_3\geq 110$, $|k_4-k_3|\leq 5$, $0\leq
k_1, k_2\leq k_3-10$, $|k_1-k_2|\leq 10$, and $f_{k_i}\in Z_{k_i}$
with $\ft^{-1}(f_{k_i})$ compactly supported (in time) in $J_0$,
$|J_0|\les 1$, $i=1,2,3$. Then
\begin{eqnarray}\label{eq:p51}
2^{k_4}\norm{\eta_{k_4}(\xi)(\tau-\omega(\xi)+i)^{-1}f_{k_1}*f_{k_2}*f_{k_3}}_{Z_{k_4}}\les
2^{(k_1+k_2)/2}\prod_{i=1}^3\norm{f_{k_i}}_{Z_{k_i}}.
\end{eqnarray}
\end{proposition}
\begin{proof}
We first divide it into three parts, according to the modulation.
\begin{eqnarray*}
&&2^{k_4}\norm{\chi_{k_4}(\xi)(\tau-\omega(\xi)+i)^{-1}f_{k_1}*f_{k_2}*f_{k_3}}_{Z_{k_4}}\\
&\leq&2^{k_4}\norm{\chi_{k_4}(\xi)\eta_{\leq k_4-1}(\tau-\omega(\xi))(\tau-\omega(\xi)+i)^{-1}f_{k_1}*f_{k_2}*f_{k_3}}_{Z_{k_4}}\\
&&+2^{k_4}\norm{\chi_{k_4}(\xi)\eta_{[k_4,2k_4]}(\tau-\omega(\xi))(\tau-\omega(\xi)+i)^{-1}f_{k_1}*f_{k_2}*f_{k_3}}_{Z_{k_4}}\\
&&+2^{k_4}\norm{\chi_{k_4}(\xi)\eta_{\geq
2k_4+1}(\tau-\omega(\xi))(\tau-\omega(\xi)+i)^{-1}f_{k_1}*f_{k_2}*f_{k_3}}_{Z_{k_4}}\\
&=&I+II+III.
\end{eqnarray*}

We consider first the contribution of $I$. Using $Y_k$ norm, then we
get from Lemma \ref{l31} (a), (c) and Lemma \ref{l34} that
\begin{eqnarray*}
I&\leq& 2^{k_4}\norm{\chi_{k_4}(\xi)\eta_{\leq
k_4-1}(\tau-\omega(\xi))(\tau-\omega(\xi)+i)^{-1}f_{k_1}*f_{k_2}*f_{k_3}}_{Y_{k_4}}\\
&\les& 2^{k_4/2}\norm{\ft^{-1}[f_{k_1}*f_{k_2}*f_{k_3}]}_{L_x^1L_t^2}\\
&\les& 2^{k_4/2}\norm{\ft^{-1}(f_{k_3})}_{L_x^\infty L_t^2}\norm{\ft^{-1}(f_{k_2})}_{L_x^2L_{t\in I_0}^\infty}\norm{\ft^{-1}(f_{k_1})}_{L_x^2L_{t\in I_0}^\infty}\\
&\les& 2^{(k_1+k_2)/2}\prod_{i=1}^3\norm{f_{k_i}}_{Z_{k_i}},
\end{eqnarray*}
which is \eqref{eq:p51} as desired.

For the contribution of $II$, we use $X_k$ norm. Then we get from
Lemma \ref{l34} that
\begin{eqnarray}
II&\leq&
2^{k_4}\norm{\chi_{k_4}(\xi)\eta_{[k_4,2k_4]}(\tau-\omega(\xi))(\tau-\omega(\xi)+i)^{-1}f_{k_1}*f_{k_2}*f_{k_3}}_{Z_{k_4}}\nonumber\\
&\leq&\sum_{k_4\leq j \leq
2k_4}2^{k_4}2^{-j/2}\norm{1_{D_{k_4,j}}(\xi,\tau)f_{k_1}*f_{k_2}*f_{k_3}}_{L^2}\nonumber\\
&\leq&\sum_{k_4\leq j \leq
2k_4}2^{k_4}2^{-j/2}\norm{\ft^{-1}(f_{k_3})}_{L_x^\infty L_t^2}\norm{\ft^{-1}(f_{k_1})}_{L_x^4 L_t^\infty}\norm{\ft^{-1}(f_{k_2})}_{L_x^4 L_t^\infty}\nonumber\\
&\les& 2^{(k_1+k_2)/4}\prod_{i=1}^3\norm{f_{k_i}}_{Z_{k_i}},
\end{eqnarray}
which is acceptable.

Finally we consider the contribution of $III$. For $j_i\geq 0$,
$i=1,2,3$, let
$f_{k_i,j_i}(\xi,\tau)=f_{k_i}(\xi,\tau)\eta_{j_i}(\tau-\omega(\xi))$.
Using $X_k$ norm, we get
\begin{equation}
III\leq \sum_{j_4\geq 2k_4+1}\sum_{j_1,j_2,j_3\geq
0}\norm{1_{D_{k_4,j_4}}(\xi,\tau)f_{k_1,j_1}*f_{k_2,j_2}*f_{k_3,j_3}}_{L^2}.
\end{equation}
Since in the area $\{|\xi_i|\in \wt{I}_{k_i}, i=1,2,3\}$, we have
$|\Omega(\xi_1,\xi_2,\xi_3)|\ll 2^{2k_4}$. By checking the support
properties of $f_{k_i,j_i}$, we get $|j_{max}-j_{sub}|\leq 5$. We
consider only  the worst case $|j_4-j_3|\leq 5$, since the other
cases are better. It follows from Corollary \ref{cor42} and Lemma
\ref{l31} (b) that
\begin{eqnarray}
III&\les& \sum_{j_3\geq 2k_4+1}\sum_{j_1,j_2\geq
0}2^{(j_1+j_2)/2}2^{(k_1+k_2)/2}\norm{f_{k_1,j_1}}_{L^2}\norm{f_{k_2,j_2}}_{L^2}\norm{f_{k_3,j_3}}_{L^2}\nonumber\\
&\les&\sum_{j_3\geq 2k_4+1}2^{j_3/4}2^{k_3-j_3}2^{(k_1+k_2)/2}2^{j_3-k_3}\norm{f_{k_1}}_{Z_{k_1}}\norm{f_{k_2}}_{Z_{k_2}}\norm{f_{k_3,j_3}}_{L^2}\nonumber\\
&\les& \sum_{j_3\geq
2k_4+1}2^{k_3-\frac{3}{4}j_3}2^{(k_1+k_2)/2}\norm{f_{k_1}}_{Z_{k_1}}\norm{f_{k_2}}_{Z_{k_2}}\norm{f_{k_3}}_{Z_{k_3}}\nonumber\\
&\les&
2^{(k_1+k_2)/4}\norm{f_{k_1}}_{Z_{k_1}}\norm{f_{k_2}}_{Z_{k_2}}\norm{f_{k_3}}_{Z_{k_3}}.
\end{eqnarray}
Therefore, we complete the proof of the proposition.
\end{proof}

This proposition suffices to control $high\times low$ interaction in
the case that the two low frequences are comparable. However, for
the case that the two low frequences are not comparable, we will
need an improvement.

\begin{proposition} Assume $k_3\geq 110$, $|k_4-k_3|\leq 5$, $0\leq k_1, k_2\leq
k_3-10$, $k_2\geq 10$, $k_1\leq k_2-5$ and $f_{k_i}\in Z_{k_i}$ with
$\ft^{-1}(f_{k_i})$ compactly supported (in time) in $J_0$,
$|J_0|\les 1$, $i=1,2,3$. Then
\begin{eqnarray}
2^{k_4}\norm{\chi_{k_4}(\xi)(\tau-\omega(\xi)+i)^{-1}f_{k_1}*f_{k_2}*f_{k_3}}_{Z_{k_4}}\leq
C 2^{(k_1+k_2)/4}\prod_{i=1}^3\norm{f_{k_i}}_{Z_{k_i}}.
\end{eqnarray}
\end{proposition}
\begin{proof}
We first observe that in this case we have
\begin{eqnarray}\label{eq:resop52}
|\Omega(\xi_1,\xi_2,\xi_3)|\sim 2^{k_3+k_2},
\end{eqnarray}
which follows from the fact that $\xi_1+\xi_2+\xi_3$ has the same
sign as $\xi_3$, $k_1\leq k_2-10$ and in the area $\{\xi_i \in
\wt{I}_{k_i}, i=1,2,3\}$
\begin{equation}
|\omega(\xi_3)-\omega(\xi_1+\xi_2+\xi_3)|\sim 2^{k_3+k_2}.
\end{equation}
Dividing it into three parts as before, we obtain
\begin{eqnarray*}
&&2^{k_4}\norm{\chi_{k_4}(\xi)(\tau-\omega(\xi)+i)^{-1}f_{k_1}*f_{k_2}*f_{k_3}}_{Z_{k_4}}\\
&\leq&2^{k_4}\norm{\chi_{k_4}(\xi)\eta_{\leq k_4-1}(\tau-\omega(\xi))(\tau-\omega(\xi)+i)^{-1}f_{k_1}*f_{k_2}*f_{k_3}}_{Z_{k_4}}\\
&&+2^{k_4}\norm{\chi_{k_4}(\xi)\eta_{[k_4,2k_4]}(\tau-\omega(\xi))(\tau-\omega(\xi)+i)^{-1}f_{k_1}*f_{k_2}*f_{k_3}}_{Z_{k_4}}\\
&&+2^{k_4}\norm{\chi_{k_4}(\xi)\eta_{\geq
2k_4+1}(\tau-\omega(\xi))(\tau-\omega(\xi)+i)^{-1}f_{k_1}*f_{k_2}*f_{k_3}}_{Z_{k_4}}\\
&=&I+II+III.
\end{eqnarray*}
For the last two terms $II,\ III$,  we can use the same argument as
for $II$, $III$ in the proof of Proposition \ref{p51}. We consider
now the first term $I$.
\begin{eqnarray*}
I&\leq&2^{k_4}\norm{\chi_{k_4}(\xi)\eta_{\leq
k_4-1}(\tau-\omega(\xi))(\tau-\omega(\xi)+i)^{-1}f_{k_1}*f_{k_2}*f_{k_3}^h}_{Z_{k_4}}\\
&&+2^{k_4}\norm{\chi_{k_4}(\xi)\eta_{\leq
k_4-1}(\tau-\omega(\xi))(\tau-\omega(\xi)+i)^{-1}f_{k_1}*f_{k_2}*f_{k_3}^l}_{Z_{k_4}}\\
&=&I_1+I_2,
\end{eqnarray*}
where
\[f_{k_3}^h=f_{k_3}(\xi,\tau)\eta_{\geq k+k_2-10}(\tau-\omega(\xi)),\quad f_{k_3}^l=f_{k_3}(\xi,\tau)\eta_{\leq k+k_2-9}(\tau-\omega(\xi)).\]

For the contribution of $I_1$, we observe first that from the
support of $f_{k_3}^h$ and the definition of $Y_{k}$, one easily get
that
\begin{equation}
\norm{f_{k_3}^h}_{X_{k_3}}\les \norm{f_{k_3}}_{Z_{k_3}}.
\end{equation}
Thus from the definition of $Y_k$, and from H\"older's inequality,
Lemma \ref{l31} (a), (c) and Lemma \ref{l34},  we get
\begin{eqnarray*}
I_1&\les&2^{k_4}\norm{\chi_{k_4}(\xi)\eta_{\leq
k_4-1}(\tau-\omega(\xi))(\tau-\omega(\xi)+i)^{-1}f_{k_1}*f_{k_2}*f_{k_3}^h}_{Y_{k_4}}\\
&\les&2^{k_4/2}\norm{\ft^{-1}[f_{k_1}*f_{k_2}*f_{k_3}^h]}_{L_x^1L_t^2}\\
&\les&2^{k_4/2}\norm{\ft^{-1}(f_{k_3}^h)}_{L_x^2L_t^2}\norm{\ft^{-1}(f_{k_1})}_{L_x^4L_t^\infty}\norm{\ft^{-1}(f_{k_2})}_{L_x^4L_t^\infty}\\
&\les&2^{k_4/2}2^{(k_1+k_2)/4}\norm{f_{k_3}^h}_{L^2}\norm{f_{k_1}}_{Z_{k_1}}\norm{f_{k_2}}_{Z_{k_2}}.
\end{eqnarray*}
Then from the fact that
\begin{eqnarray}
2^{k_4/2}\norm{f_{k_3}^h}_{L^2}&\les& \sum_{j\geq
k_4+k_2-10}2^{k_4/2}\norm{f_{k_3}^h\eta_{j}(\tau-\omega(\xi))}_{L^2}\nonumber\\
&\les&\norm{f_{k_3}^h}_{X_{k_3}}\les \norm{f_{k_3}}_{Z_{k_3}}
\end{eqnarray}
we conclude the proof for $I_1$.

We consider now the contribution of $I_2$. Let
$f_{k_i,j_i}(\xi,\tau)=f_{k_i}(\xi,\tau)\eta_{j_i}(\tau-\omega(\xi))$,
$j_i\geq 0$, $i=1,2,3$. Using $X_k$ norm, we get
\begin{equation}
I_2\leq \sum_{j_4\leq k_4-1}\sum_{j_3\leq
k_4+k_2-9}\sum_{j_1,j_2\geq
0}2^{k_4}2^{-j_4/2}\norm{1_{D_{k_4,j_4}}(\xi,\tau)f_{k_1,j_1}*f_{k_2,j_2}*f_{k_3,j_3}}_{L^2}.\nonumber
\end{equation}
By checking the support properties of the functions $f_{k_i,j_i}\
(i=1,2,3)$ and from \eqref{eq:resop52}, we easily get that
$1_{D_{k_4,j_4}}(\xi,\tau)f_{k_1,j_1}*f_{k_2,j_2}*f_{k_3,j_3}\equiv
0$ unless
\begin{eqnarray*}
\left \{
\begin{array}{l}
j_1, j_2 \geq k_3+k_2-10, |j_1-j_2|\leq 5; \mbox{ or }\\
|j_1-k_3-k_2|\leq 5, j_2\leq k_3+k_2-10; \mbox{ or }
|j_2-k_3-k_2|\leq 5, j_1\leq k_3+k_2-10.
\end{array}
\right.
\end{eqnarray*}

{\bf Case 1.} $j_1, j_2 \geq k_3+k_2-10, |j_1-j_2|\leq 5$.

It follows from Corollary \ref{cor42} (b) and Lemma \ref{l31} (b)
that
\begin{eqnarray}
I_2&\les& \sum_{j_3,j_4\leq 2k_4}\sum_{j_1,j_2\geq
0}2^{k_4}2^{-j_4/2}2^{(j_1+j_2+j_3+j_4)/2}2^{-j_1/2}2^{-k_3/2}2^{k_1/2}\prod_{i=1}^3\norm{f_{k_i,j_i}}_2\nonumber\\
&\les& \sum_{j_1,j_2\geq 0}k_4^2
2^{j_2/2}2^{k_3/2}2^{k_1/2}\prod_{i=1}^2\norm{f_{k_i,j_i}}_2
\norm{f_{k_3}}_{Z_{k_3}}\nonumber\\
&\les& \sum_{j_1,j_2\geq
0}k_4^2 2^{j_2/2}2^{k_3/2}2^{k_1/2}2^{k_1+k_2-j_1-j_2}\prod_{i=1}^2(2^{j_i-k_i}\norm{f_{k_i,j_i}}_2) \norm{f_{k_3}}_{Z_{k_3}}\nonumber\\
&\les&2^{(k_1+k_2)/4}\prod_{i=1}^3\norm{f_{k_i}}_{Z_{k_i}},
\end{eqnarray}
which is acceptable.

{\bf Case 2.} $|j_1-k_3-k_2|\leq 5, j_2\leq k_3+k_2-10; \mbox{ or }
|j_2-k_3-k_2|\leq 5, j_1\leq k_3+k_2-10.$

We consider only the worse case $|j_2-k_3-k_2|\leq 5,\ j_1\leq
k_3+k_2-10$. It follows from Corollary (b) that
\begin{eqnarray*}
I_2&\les& \sum_{j_1,j_3,j_4\leq 2k_4}\sum_{j_2\geq
0}2^{k_4}2^{(j_1+j_3)/2}2^{-k_3/2}2^{k_2/2}\prod_{i=1}^3\norm{f_{k_i,j_i}}_2\\
&\les&  k_4^32^{k_4}2^{k_2-j_2}2^{-k_3/2}2^{k_2/2}\prod_{i=1}^3\norm{f_{k_i}}_{Z_{k_i}}\\
&\les&2^{(k_1+k_2)/4}\prod_{i=1}^3\norm{f_{k_i}}_{Z_{k_i}}.
\end{eqnarray*}
Therefore, we complete the proof of the proposition.
\end{proof}

\begin{proposition} Assume $k_3\geq 110$, $|k_4-k_3|\leq 5$, $k_3-10\leq k_2\leq
k_3$, $0\leq k_1\leq k_2-10$ and $f_{k_i}\in Z_{k_i}$ with
$\ft^{-1}(f_{k_i})$ compactly supported (in time) in $J_0$,
$|J_0|\les 1$, $i=1,2,3$. Then
\begin{eqnarray}
2^{k_4}\norm{\chi_{k_4}(\xi)(\tau-\omega(\xi)+i)^{-1}f_{k_1}*f_{k_2}*f_{k_3}}_{Z_{k_4}}\les
2^{(k_1+k_2)/4}\prod_{i=1}^3\norm{f_{k_i}}_{Z_{k_i}}.
\end{eqnarray}
\end{proposition}
\begin{proof}
We first observe that this case corresponds to an integration in the
area $\{|\xi_i|\in I_{k_i},\ i=1,2,3\}\cap \{|\xi_1+\xi_2+\xi_3|\in
I_{k_4}\}$, where we have
\begin{equation}
|\Omega(\xi_1,\xi_2,\xi_3)|\sim 2^{2k_3}.
\end{equation}
Let
$f_{k_i,j_i}(\xi,\tau)=f_{k_i}(\xi,\tau)\eta_{j_i}(\tau-\omega(\xi))$,
$j_i\geq 0$ and $i=1,2,3$. Using $X_k$ norm, we get
\begin{eqnarray}\label{eq:p531}
&&2^{k_4}\norm{\chi_{k_4}(\xi)(\tau-\omega(\xi)+i)^{-1}f_{k_1}*f_{k_2}*f_{k_3}}_{Z_{k_4}}\nonumber\\
&\les& \sum_{j_1,j_2,j_3,j_4\geq
0}2^{k_4}2^{-j_4/2}(1+2^{(j_4-2k_4)/2})\norm{1_{D_{k_4,j_4}}(\xi,\tau)f_{k_1,j_1}*f_{k_2,j_2}*f_{k_3,j_3}}_{L^2}.\nonumber
\end{eqnarray}
From the support properties of the functions $f_{k_i,j_i}$,
$i=1,2,3$, it is easy to see that
$1_{D_{k_4,j_4}}(\xi,\tau)f_{k_1,j_1}*f_{k_2,j_2}*f_{k_3,j_3}\equiv
0$ unless
\begin{eqnarray*}
\left \{
\begin{array}{l}
j_{max}, j_{sub} \geq 2k_3-10, |j_{max}-j_{sub}|\leq 5; \mbox{ or }\\
|j_{max}-2k_3|\leq 5, j_{sub}\leq 2k_3-10.
\end{array}
\right.
\end{eqnarray*}

{\bf Case 1.} $j_{max}, j_{sub} \geq 2k_3-10, |j_{max}-j_{sub}|\leq
5$.

It follows from Corollary \ref{cor42} (a) that the right-hand side
of \eqref{eq:p531} is bounded by
\begin{eqnarray}\label{eq:p53c1}
&&\sum_{j_1,j_2,j_3,j_4\geq
0}2^{k_4}2^{(j_1+j_2+j_3)/2}(1+2^{(j_4-2k_4)/2})\nonumber\\
&&\qquad
2^{-(j_{sub}+j_{max})/2}2^{(k_1+k_2)/2}\prod_{i=1}^3\norm{f_{k_i,j_i}}_2.
\end{eqnarray}
It suffices to consider the worst case $j_3,j_4=j_{max}, j_{sub}$.
We get from Lemma \ref{l31} (b) that  \eqref{eq:p53c1} is bounded by
\begin{eqnarray}
\sum_{j_3\geq
2k_3-10}2^{k_4}2^{-\frac{3}{4}j_3}2^{(k_1+k_2)/2}\prod_{i=1}^3\norm{f_{k_i}}_{Z_{k_i}}\les2^{(k_1+k_2)/4}\prod_{i=1}^3\norm{f_{k_i}}_{Z_{k_i}}.
\end{eqnarray}

{\bf Case 2.} $|j_{max}-2k_3|\leq 5, j_{sub}\leq 2k_3-10$.

From Corollary \ref{cor42} (c), we get that the right-hand side of
\eqref{eq:p531} is bounded by
\begin{eqnarray}
&&\sum_{j_1,j_2,j_3,j_4\geq
0}2^{k_4}2^{(j_1+j_2+j_3)/2}2^{-j_{max}/2}\prod_{i=1}^3\norm{f_{k_i,j_i}}_2\nonumber\\
&&\les 2^{(k_1+k_2)/4}\prod_{i=1}^3\norm{f_{k_i}}_{Z_{k_i}},
\end{eqnarray}
where we used Lemma \ref{l31} (b). Thus, we complete the proof of
the proposition.
\end{proof}

\begin{proposition} Assume $k_3\geq 110$, $|k_4-k_3|\leq 5$, $k_3-30\leq k_1,\ k_2
\leq k_3$, and $f_{k_i}\in Z_{k_i}$ with $\ft^{-1}(f_{k_i})$
compactly supported (in time) in $J_0$, $|J_0|\les 1$, $i=1,2,3$.
Then
\begin{eqnarray}
2^{k_4}\norm{\chi_{k_4}(\xi)(\tau-\omega(\xi)+i)^{-1}f_{k_1}*f_{k_2}*f_{k_3}}_{Z_{k_4}}\leq
C 2^{k_3}\prod_{i=1}^3\norm{f_{k_i}}_{Z_{k_i}}.
\end{eqnarray}
\end{proposition}
\begin{proof}
First we divide it into two parts.
\begin{eqnarray*}
&&2^{k_4}\norm{\chi_{k_4}(\xi)(\tau-\omega(\xi)+i)^{-1}f_{k_1}*f_{k_2}*f_{k_3}}_{Z_{k_4}}\\
&\les& 2^{k_4}\norm{\chi_{k_4}(\xi)\eta_{\leq 2k_4+20} (\tau-\omega(\xi))(\tau-\omega(\xi)+i)^{-1}f_{k_1}*f_{k_2}*f_{k_3}}_{Z_{k_4}}\\
&&+2^{k_4}\norm{\chi_{k_4}(\xi)\eta_{\geq 2k_4+21}
(\tau-\omega(\xi))(\tau-\omega(\xi)+i)^{-1}f_{k_1}*f_{k_2}*f_{k_3}}_{Z_{k_4}}\\
&=&I+II.
\end{eqnarray*}

We consider first the contribution of the first term $I$. Using the
$X_k$ norm and Lemma \ref{l34}, then we get
\begin{eqnarray}
I&\les& 2^{k_4}\sum_{j_4\geq
0}^{2k_4+20}2^{-j_4/2}\norm{1_{D_{k_4,j_4}}(\xi,\tau)f_{k_1}*f_{k_2}*f_{k_3}}_{L^2}\nonumber\\
&\les&2^{k_4}\prod_{i=1}^3\norm{\ft^{-1}(f_{k_i})}_{L^6}\les
2^{k_4}\prod_{i=1}^3\norm{f_{k_i}}_{Z_{k_i}}.
\end{eqnarray}

We consider now the contribution of the second term $II$. Let
$f_{k_i,j_i}(\xi,\tau)=f_{k_i}(\xi,\tau)\eta_{j_i}(\tau-\omega(\xi))$,
$j_i\geq 0$, $i=1,2,3$. Using the $X_k$ norm, we get
\begin{eqnarray}
II&\les&\sum_{j_4\geq 2k_4+20}\sum_{j_1,j_2,j_3\geq
0}\norm{1_{D_{k_4,j_4}}(\xi,\tau)f_{k_1,j_1}*f_{k_2,j_2}*f_{k_3,j_3}}_{L^2}.
\end{eqnarray}
Since in the area $\{|\xi_i|\in I_{k_i},\ i=1,2,3\}$ we have
$|\Omega(\xi_1,\xi_2,\xi_3)|\les 2^{2k_3}$, by checking the support
properties of the functions $f_{k_i,j_i}$, $i=1,2,3$, we get
$|j_{max}-j_{sub}|\leq 5$ and $j_{sub}\geq 2k_3+10$. From symmetry,
we assume $j_3,j_4=j_{max}, j_{sub}$, then we get
\begin{eqnarray}
II&\les&\sum_{j_4\geq 2k_4+20}\sum_{j_1,j_2,j_3\geq
0}2^{(j_1+j_2)/2}2^{k_3}2^{k_3-j_3}2^{j_3-k_3}\prod_{i=1}^3\norm{f_{k_i,j_i}}_2\nonumber\\
&\les&\sum_{j_3\geq
2k_4+20}2^{2k_3}2^{-j_3/2}\prod_{i=1}^3\norm{f_{k_i}}_{Z_{k_i}}\les2^{k_4}\prod_{i=1}^3\norm{f_{k_i}}_{Z_{k_i}}.
\end{eqnarray}
Therefore we complete the proof of the proposition.
\end{proof}

We consider now the case which corresponds to $high\times high $
interactions. This case is better than $high\times low$ interaction
case.
\begin{proposition} Assume $k_1\geq 110$, $|k_1-k_2|\leq 5$, $0\leq k_3\leq k_1+10,\ 0\leq k_4\leq
k_1-10$, and $f_{k_i}\in Z_{k_i}$ with $\ft^{-1}(f_{k_i})$ compactly
supported (in time) in $J_0$, $|J_0|\les 1$, $i=1,2,3$. Then
\begin{eqnarray}
2^{k_4}\norm{\eta_{k_4}(\xi)(\tau-\omega(\xi)+i)^{-1}f_{k_1}*f_{k_2}*f_{k_3}}_{Z_{k_4}}
\leq C {k^4_1}\prod_{i=1}^3\norm{f_{k_i}}_{Z_{k_i}}
\end{eqnarray}
\end{proposition}

\begin{proof}
We divide the argument into two cases. Let
$f_{k_i,j_i}(\xi,\tau)=f_{k_i}(\xi,\tau)\eta_{j_i}(\tau-\omega(\xi))$,
$j_i\geq 0$, $i=1,2,3$. Using $X_k$ norm, then we get
\begin{eqnarray}\label{eq:hh}
&&2^{k_4}\norm{\chi_{k_4}(\xi)(\tau-\omega(\xi)+i)^{-1}f_{k_1}*f_{k_2}*f_{k_3}}_{Z_{k_4}}\nonumber\\
&\les& \sum_{j_i\geq
0}2^{k_4}2^{-j_4/2}(1+2^{(j_4-2{k_4})/2})\norm{1_{D_{k_4,j_4}}(\xi,\tau)f_{k_1,j_1}*f_{k_2,j_2}*f_{k_3,j_3}}_{L^2}.\qquad
\end{eqnarray}

{\bf Case 1.} $\max(j_1,j_2,j_3,j_4)\leq 2k_1+20$.

It follows from Corollary \ref{cor42} (d) that the right-hand side
of \eqref{eq:hh} is bounded by
\begin{eqnarray}
\sum_{j_i\geq
0}2^{k_4}(1+2^{(j_4-2{k_4})/2})2^{(j_1+j_2+j_3)/2}2^{-k_1}\prod_{i=1}^3\norm{f_{k_i,j_i}}_{L^2}\les{k^4_1}
\prod_{i=1}^3\norm{f_{k_i}}_{Z_{k_i}},
\end{eqnarray}
where we used Lemma \ref{l31} (b).

{\bf Case 2.} $\max(j_1,j_2,j_3,j_4)\geq 2k_1+20$.

By checking the support properties, we get $|j_{max}-j_{sub}|\leq
5$. We consider only the worst case $j_1,j_4=j_{max},j_{sub}$. It
follows from Corollary \ref{cor42} (a) and Lemma \ref{l31} (b) that
the right side of \eqref{eq:hh} is bounded by
\begin{eqnarray}
\sum_{j_i\geq
0}2^{k_4}2^{-j_4/2}(1+2^{(j_4-2{k_4})/2})2^{(j_2+j_3)/2}2^{k_4}\prod_{i=1}^3\norm{f_{k_i,j_i}}_{L^2}\les{k_1^4}\prod_{i=1}^3\norm{f_{k_i}}_{Z_{k_i}}.
\end{eqnarray}
Therefore, we complete the proof of the proposition.
\end{proof}

The next proposition is used to control $low \times low$
interactions. This interaction is easy to control.

\begin{proposition} Assume $0\leq k_1, k_2, k_3, k_4\leq 120$,  and $f_{k_i}\in Z_{k_i}$, $i=1,2,3$. Then
\begin{eqnarray}
2^{k_4}\norm{\eta_{k_4}(\xi)(\tau-\omega(\xi)+i)^{-1}f_{k_1}*f_{k_2}*f_{k_3}}_{Z_{k_4}}\leq
C\prod_{i=1}^3\norm{f_{k_i}}_{Z_{k_i}}.
\end{eqnarray}
\end{proposition}
\begin{proof}
Let
$f_{k_i,j_i}(\xi,\tau)=f_{k_i}(\xi,\tau)\eta_{j_i}(\tau-\omega(\xi))$,
$j_i\geq 0$, $i=1,2,3$. Using $X_k$ norm, Corollary \ref{cor42} (a)
and Lemma \ref{l31} (b), then we get
\begin{eqnarray*}
&&2^{k_4}\norm{\chi_{k_4}(\xi)(\tau-\omega(\xi)+i)^{-1}f_{k_1}*f_{k_2}*f_{k_3}}_{Z_{k_4}}\\
&\les& \sum_{j_1,j_2,j_3,j_4\geq
0}2^{(j_{min}+j_{thd})/2}2^{(k_{min}+k_{thd})/2}\prod_{i=1}^3\norm{f_{k_i,j_i}}_{L^2}\\
&\les&2^{(k_{min}+k_{thd})/2}\prod_{i=1}^3\norm{f_{k_i}}_{Z_{k_i}},
\end{eqnarray*}
since for the case $j_{max}\geq 200$ we have $|j_{max}-j_{sub}|\leq
5$ by checking the support properties of the functions
$f_{k_i,j_i}$, $i=1,2,3$.
\end{proof}

Finally we present two counterexamples. The first one shows why we
use a $l^1$-type $X^{s,b}$ structure. The other one shows a
logarithmic divergence if we only use $X_k$ which is the reason for
us applying $Y_k$ structure. See also the similar phenomenon in
\cite{IK2} for the complex-valued Benjamin-Ono equation.

\begin{proposition}\label{countertrilinear}
Assume $k\geq 200$. Then there exist $f_{1}\in X_1,\ f_k\in X_k$
such that
\begin{eqnarray}
2^{k}\norm{\eta_k(\xi)(\tau-\omega(\xi)+i)^{-1}f_1*f_1*f_k}_{X_k}\ges
\  k \norm{f_1}_{X_1}\norm{f_1}_{X_1}\norm{f_k}_{X_k}.
\end{eqnarray}
\end{proposition}
\begin{proof}
From the proof of Proposition \ref{p51}, we easily see that the
worst interaction comes from the case that largest frequency
component has a largest modulation. So we construct this case
explicitly. Let $I=[1/2,1]$, and take
\[f_1(\xi,\tau)=\chi_{I}(\xi)\eta_1(\tau-\omega(\xi)),\ f_k(\xi,\tau)=\chi_{I_k}(\xi)\eta_k(\tau-\omega(\xi)).\]
From definition, we easily get $\norm{f_1}_{X_1}\sim 1$ and
$\norm{f_k}_{X_k}\sim 2^{3k/2}$ and
\[2^{k}\norm{\eta_k(\xi)(\tau-\omega(\xi)+i)^{-1}f_1*f_1*f_k}_{X_k}\ges 2^{k}\sum_{j=0}^{k/2}2^{-j/2}\norm{1_{D_{k,j}}\cdot f_1*f_1*f_k}_{L_{\xi,\tau}^2}.\]
On the other hand, we have for $j\leq k/2$
\begin{eqnarray*}
&&1_{D_{k,j}}(\xi,\tau)\cdot f_1*f_1*f_k\\
&=&\int
f_1(\xi_1,\tau_1)f_2(\xi_2,\tau_2)f_k(\xi-\xi_1-\xi_2,\tau-\tau_1-\tau_2)d\xi_1d\xi_2d\tau_1d\tau_2\\
&=&\int
\chi_I(\xi_1)\chi_I(\xi_2)\eta_1(\tau_1)\eta_1(\tau_2)\chi_{I_k}(\xi-\xi_1-\xi_2)\\
&&\cdot \eta_k(
\tau-\tau_1-\tau_2-\omega(\xi_1)-\omega(\xi_2)-\omega(\xi-\xi_1-\xi_2))d\xi_1d\xi_2d\tau_1d\tau_2\\
&\ges&
\chi_{[\frac{2^{10}-1}{2^{10}}2^k,\frac{2^{10}+1}{2^{10}}2^k]}(\xi)\eta_j(\tau-\omega(\xi)).
\end{eqnarray*}
Therefore, we get
\begin{eqnarray}\label{eq:logdiv}
2^{k}\sum_{j=0}^{k/2}2^{-j/2}\norm{1_{D_{k,j}}\cdot
f_1*f_1*f_k}_{L_{\xi,\tau}^2}\ges \ k2^{3k/2},
\end{eqnarray}
which completes the proof of the proposition.
\end{proof}

\begin{proposition}\label{counterXsb}
For any $s\in \R$, there doesn't exists $b\in \R$ such that
\begin{equation}\label{eq:trilinearXsb}
\norm{\partial_x(uvw)}_{X^{s,b-1}}\les\
\norm{u}_{X^{s,b}}\norm{v}_{X^{s,b}}\norm{w}_{X^{s,b}}.
\end{equation}
\end{proposition}
\begin{proof}
It is easy to see that the counterexample in the proof of
Proposition \ref{countertrilinear} shows that
\eqref{eq:trilinearXsb} doesn't hold for $b=1/2$ with a $k^{1/2}$
divergence in \eqref{eq:logdiv}. We assume now $b\neq 1/2$. By using
Plancherel's equality, we get that \eqref{eq:trilinearXsb} is
equivalent to
\begin{eqnarray}\label{eq:trilinearXsbequiv}
&&\norm{\frac{\jb{\xi}^s\xi}{\jb{\tau-\omega(\xi)}^{1-b}}\int
\frac{u(\xi_1,\tau_1)}{\jb{\xi_1}^s \jb{\tau_1-\omega(\xi_1)}^b}\frac{v(\xi_2,\tau_2)}{\jb{\xi_2}^s \jb{\tau_2-\omega(\xi_2)}^b}\nonumber\\
&&\quad \cdot \frac{w(\xi-\xi_1-\xi_2,\tau-\tau_1-\tau_2)}{\jb{\xi-\xi_1-\xi_2}^s \jb{\tau-\tau_1-\tau_2-\omega(\xi-\xi_1-\xi_2)}^b}d\tau_1 d\tau_2d\xi_1d\xi_2}_{L_{\xi,\tau}^2}\nonumber\\
&&\les \  \norm{u}_{L^2}\norm{v}_{L^2}\norm{w}_{L^2}.
\end{eqnarray}
Fix any dyadic number $N\gg 1$. Let
\[A=\{1/2\leq \xi\leq 10,\ |\tau|\leq
1\} \mbox{ and } B=\{N/2\leq \xi\leq 2N,\ |\tau|\leq 2^{10}\}.\]
Take
\[u(\xi,\tau)=v(\xi,\tau)=\chi_{A}(\xi,\tau-\omega(\xi)),\ w(\xi,\tau)=\chi_{B}(\xi,\tau-\omega(\xi)).\]
We easily see that $\norm{u}_{L_2}=\norm{v}_{L_2}\sim 1$ and
$\norm{w}_{L^2}\sim N^{1/2}$. Denote
$f(\xi,\tau)=u*v*w(\xi,\tau+\omega(\xi))$. Then we have
\begin{eqnarray*}
f(\xi,\tau)&=&\int
u(\xi_1,\tau_1)v(\xi_2,\tau_2)w(\xi-\xi_1-\xi_2,\tau+\omega(\xi)-\tau_1-\tau_2)d\xi_1d\xi_2d\tau_1d\tau_2\\
&=&\int \chi_{\leq
2^{10}}(\tau-\tau_1-\tau_2+\omega(\xi)-\omega(\xi-\xi_1-\xi_2)-\omega(\xi_1)-\omega(\xi_2))\\
&&\chi_A(\xi_1,\tau_1)\chi_A(\xi_2,\tau_2)\chi_{[N/2,2N]}(\xi-\xi_1-\xi_2)d\xi_1d\xi_2d\tau_1d\tau_2\\
&=&\int \chi_{\leq 2^{10}}(\tau-\tau_1-\tau_2+2(\xi_1+\xi_2)\xi+(\xi_1-\xi_2)^2-\omega(\xi_1)-\omega(\xi_2))\\
&&\chi_A(\xi_1,\tau_1)\chi_A(\xi_2,\tau_2)\chi_{[N/2,2N]}(\xi-\xi_1-\xi_2)d\xi_1d\xi_2d\tau_1d\tau_2.
\end{eqnarray*}
Therefore, fixing $M\gg 1$, we get for any $(\xi, \tau) \in
[(M-1)N/M, (M+1)N/M]\times [-8N,-4N]$, then $\tau=-C_0\xi$ for some
$2\leq C_0\leq 9$ and
\begin{eqnarray*}
f(\xi,\tau)\ges
\int\chi_A(\xi_1,\tau_1)\chi_A(\xi_2,\tau_2)\chi_{|\xi_1+\xi_2-C_0|\les
N^{-1}}d\xi_1d\xi_2d\tau_1d\tau_2\ges N^{-1}.
\end{eqnarray*}
Thus we see that the left-hand side of \eqref{eq:trilinearXsbequiv}
is larger than $N^b$, while the right-hand side is $N^{1/2}$, which
implies $b< 1/2$.

Similarly, by taking $B'=\{N/2\leq \xi\leq 2N,\ N\leq |\tau|\leq
N\}$ as before, we obtain that $b>1/2$. Therefore we complete the
proof of the proposition.
\end{proof}

\section{Proof of Theorem \ref{t11}}

In this section we devote to prove Theorem \ref{t11} by using the
standard fixed-point machinery. From Duhamel's principle, we get
that the equation \eqref{eq:mBO} is equivalent to the following
integral equation:
\begin{eqnarray}\label{eq:inteq}
u=W(t)\phi+\int_0^t W(t-t')(\partial_x(u^3)(t'))dt'.
\end{eqnarray}
We will mainly work on the following truncated version
\begin{eqnarray}\label{eq:truninteq}
u=\psi(t)W(t)\phi+\psi(t)\int_0^t
W(t-t')(\partial_x[(\psi(t')u)^3](t'))dt',
\end{eqnarray}
where $\psi(t)=\eta_0(t)$ is a smooth cut-off function. Then we
easily see that if $u$ is a solution to \eqref{eq:truninteq} on
$\R$, then $u$ solves \eqref{eq:inteq} on $t\in [-1,1]$. Our first
lemma is on the estimate for the linear solution.

\begin{lemma}\label{l61}
If $s\geq 0$ and $\phi \in {H}^s$ then
\begin{equation}
\norm{\psi(t)\cdot (W(t)\phi)}_{F^{s}}\leq C \norm{\phi}_{{H}^s}.
\end{equation}
\end{lemma}
\begin{proof}
A direct computation shows that
\begin{eqnarray*}
\ft[\psi(t)\cdot
(W(t)\phi)](\xi,\tau)=\widehat{\phi}(\xi)\widehat{\psi}(\tau-\omega(\xi)).
\end{eqnarray*}
In view of definition, it suffices to prove that if $k\in \Z_+$ then
\begin{equation}\label{eq:l61}
\norm{\eta_k(\xi)\widehat{\phi}(\xi)\widehat{\psi}(\tau-\omega(\xi))}_{Z_k}\leq
C \norm{\eta_k(\xi)\widehat{\phi}(\xi)}_{L^2}.
\end{equation}
Indeed, from definition we have
\begin{eqnarray*}
\norm{\eta_k(\xi)\widehat{\phi}(\xi)\widehat{\psi}(\tau-\omega(\xi))}_{Z_k}&\leq&
\norm{\eta_k(\xi)\widehat{\phi}(\xi)\widehat{\psi}(\tau-\omega(\xi))}_{X_k}\\
&\leq&C\sum_{j=0}^\infty 2^{j}
\norm{\eta_k(\xi)\widehat{\phi}(\xi)}_{L^2}
\norm{\eta_j(\tau)\widehat{\psi}(\tau)}_{L^2}\\
&\leq&C\norm{\eta_k(\xi)\widehat{\phi}(\xi)}_{L^2},
\end{eqnarray*}
which is \eqref{eq:l61} as desired.
\end{proof}

Next lemma is on the estimate for the retarded linear term. These
estimates were also used in \cite{IK}. The only difference is that
here we don't have special structure for the low frequency.

\begin{lemma}\label{l62}
If $l,s\geq 0$ and $u \in \Sch(\R\times \R)$ then
\begin{equation}
\normo{\psi(t)\cdot \int_0^tW(t-s)(u(s))ds}_{F^{s}}\leq C
\norm{u}_{N^{s}}.
\end{equation}
\end{lemma}
\begin{proof}
A straightforward computation shows that
\begin{eqnarray*}
&&\ft\left[\psi(t)\cdot
\int_0^tW(t-s)(u(s))ds\right](\xi,\tau)\\
&&=c\int_\R
\ft(u)(\xi,\tau')\frac{\widehat{\psi}(\tau-\tau')-\widehat{\psi}(\tau-\omega(\xi))}{\tau'-\omega(\xi)}d\tau'.
\end{eqnarray*}
For $k\in \Z_+$ let
$f_k(\xi,\tau')=\ft(u)(\xi,\tau')\chi_k(\xi)(\tau'-\omega(\xi)+i)^{-1}$.
For $f_k\in Z_k$ let
\begin{eqnarray*}
T(f_k)(\xi,\tau)=\int_\R
f_k(\xi,\tau')\frac{\widehat{\psi}(\tau-\tau')-\widehat{\psi}(\tau-\omega(\xi))}{\tau'-\omega(\xi)}(\tau'-\omega(\xi)+i)d\tau'.
\end{eqnarray*}
In view of the definitions, it suffices to prove that
\begin{equation}
\norm{T}_{Z_k\rightarrow Z_k}\leq C \mbox{ uniformly in } k\in Z_+,
\end{equation}
which follows from the proof of Lemma 5.2 in \cite{IK}.
\end{proof}

We prove a trilinear estimate in the following proposition which is
an important component for using fixed-point argument.
\begin{proposition}\label{p63}
Let $s\geq 1/2$. Then
\begin{eqnarray}
\norm{\partial_x(\psi(t)^3uvw)}_{N^{s}}&\les&
\norm{u}_{F^{s}}\norm{v}_{F^{\frac{1}{2}}}\norm{w}_{F^{\frac{1}{2}}}\nonumber\\
&&+\norm{u}_{F^{\frac{1}{2}}}\norm{v}_{F^{s}}\norm{w}_{F^{\frac{1}{2}}}+\norm{u}_{F^{\frac{1}{2}}}\norm{v}_{F^{\frac{1}{2}}}\norm{w}_{F^{s}}.
\end{eqnarray}
\end{proposition}
\begin{proof}
For the simplicity of notation, we write $u=\psi(t)u$, $v=\psi(t)v$
and $w=\psi(t)w$. In view of definition, we get
\begin{eqnarray*}
\norm{\partial_x(uvw)}_{N^{s}}^2=\sum_{k_4=0}^{\infty}2^{2sk_4}\norm{\eta_{k_4}(\xi)(\tau-\omega(\xi)+i)^{-1}\ft
(\partial_x(uvw))}_{Z_k}^2.
\end{eqnarray*}
Setting $f_{k_1}=\eta_{k_1}(\xi)\ft(u)(\xi,\tau)$,
$f_{k_2}=\eta_{k_2}(\xi)\ft(v)(\xi,\tau)$, and
$f_{k_3}=\eta_{k_3}(\xi)\ft(w)(\xi,\tau)$, for $k_1, k_2, k_3\in
\Z_+$, then we get
\begin{eqnarray*}
&&2^{k_4}\norm{\eta_{k_4}(\xi)(\tau-\omega(\xi)+i)^{-1}\ft
(uvw)}_{Z_{k_4}}\\
&&\les \sum_{k_1,k_2,k_3\in
\Z_+}2^{k_4}\norm{\eta_{k_4}(\xi)(\tau-\omega(\xi)+i)^{-1}f_{k_1}*f_{k_2}*f_{k_3}}_{Z_{k_4}}.
\end{eqnarray*}
From symmetry it suffices to bound
\begin{eqnarray*}
\sum_{0\leq k_1\leq k_2\leq
k_3}2^{k_4}\norm{\eta_{k_4}(\xi)(\tau-\omega(\xi)+i)^{-1}f_{k_1}*f_{k_2}*f_{k_3}}_{Z_{k_4}}.
\end{eqnarray*}
Dividing the summation into the several parts, we get
\begin{eqnarray}\label{eq:trilinear}
&&\sum_{k_1\leq k_2\leq
k_3}2^{k_4}\norm{\eta_{k_4}(\xi)(\tau-\omega(\xi)+i)^{-1}f_{k_1}*f_{k_2}*f_{k_3}}_{Z_{k_4}}\nonumber\\
&\leq& \sum_{j=1}^6\sum_{(k_1,k_2,k_3,k_4)\in A_j}
2^{k_4}\norm{\eta_{k_4}(\xi)(\tau-\omega(\xi)+i)^{-1}f_{k_1}*f_{k_2}*f_{k_3}}_{Z_{k_4}},
\end{eqnarray}
where we denote
\begin{eqnarray*}
&&A_1=\{0\leq k_1\leq k_2\leq k_3-10, k_3\geq 110, |k_4-k_3|\leq 5, |k_1-k_2|\leq 10 \};\\
&&A_2=\{0\leq k_1\leq k_2\leq k_3-10, k_3\geq 110, |k_4-k_3|\leq 5,
k_1\leq k_2-5 \};\\
&&A_3=\{0\leq k_1\leq k_2\leq k_3, k_2\geq k_3-10, k_3\geq 110,
|k_4-k_3|\leq 5, k_1\leq k_2-10 \};\\
&&A_4=\{0\leq k_1\leq k_2\leq k_3, k_1\geq k_3-30, k_3\geq 110,
|k_4-k_3|\leq 5\};\\
&&A_5=\{0\leq k_1\leq k_2\leq k_3, k_4\leq k_3-10, k_3\geq 110,
|k_2-k_3|\leq 5\};\\
&&A_6=\{0\leq k_1\leq k_2\leq k_3, \max(k_3,k_4)\leq 120\}.
\end{eqnarray*}
Noting that $\ft^{-1}(f_{k_i})(x,t)$ is supported in $\R\times I$
with $|I|\les 1$, we will apply Proposition 5.1-5.6 obtained in the
last section to bound the six terms in \eqref{eq:trilinear}. For
example, for the first term, from Proposition 5.1, we have
\begin{eqnarray*}
&&\normb{2^{sk_4}\sum_{k_i\in
A_1}2^{k_4}\norm{\eta_{k_4}(\xi)(\tau-\omega(\xi)+i)^{-1}f_{k_1}*f_{k_2}*f_{k_3}}_{Z_{k_4}}}_{l_{k_4}^2}\\
&\leq& C\normb{2^{sk_4}\sum_{k_i\in
A_1}2^{(k_1+k_2)/2}\norm{f_{k_1}}_{Z_{k_1}}\norm{f_{k_2}}_{Z_{k_2}}\norm{f_{k_3}}_{Z_{k_3}}}_{l_{k_4}^2}\\
&\leq& \norm{u}_{F^{1/2}}\norm{v}_{F^{1/2}}\norm{w}_{F^s}.
\end{eqnarray*}
For the other terms we can handle them in the similar ways.
Therefore we complete the proof of the proposition.
\end{proof}

Now we prove Theorem \ref{t11}. To begin with, we renormalize the
data a bit via scaling. By the scaling \eqref{eq:scaling}, we see
that if $s\geq 1/2$
\begin{eqnarray*}
&&\norm{\phi_\lambda}_{L^2}=\norm{\phi}_{L^2},\\
&&\norm{\phi_\lambda}_{\dot{H}^s}=\lambda^{-s}\norm{\phi}_{\dot{H}^s}.
\end{eqnarray*}
From the assumption $\norm{\phi}_{L^2}\ll 1$, thus we can first
restrict ourselves to considering \eqref{eq:mBO} with data $\phi$
satisfying
\begin{equation}
\norm{\phi}_{H^s}=r\ll 1.
\end{equation}
This indicates the reason why we assume that $\norm{\phi}_{L^2}\ll
1$.

Define the operator
\begin{eqnarray*}
\Phi_{\phi}(u)=\psi(t)W(t)\phi+\psi(t)\int_0^t
W(t-t')(\partial_x((\psi(t')u)^3)(t'))dt',
\end{eqnarray*}
and we will prove that $\Phi_\phi (\cdot)$ is a contraction  mapping
from
\begin{equation}
{\mathcal{B}}=\{w\in F^s:\ \norm{w}_{F^s}\leq 2cr\}
\end{equation}
into itself. From Lemma \ref{l61}, \ref{l62} and Proposition
\ref{p63} we get if $w\in \mathcal{B}$, then
\begin{eqnarray}
\norm{\Phi_\phi(w)}_{F^s}&\leq&
c\norm{\phi}_{H^s}+\norm{\partial_x(\psi(t)^3w^3(\cdot, t))}_{N^s}\nonumber\\
&\leq& cr+c\norm{w}_{F^s}^3\leq cr+c(2cr)^3\leq 2cr,
\end{eqnarray}
provided that $r$ satisfies $8c^3r^2\leq 1/2$. Similarly, for $w,
h\in \mathcal{B}$
\begin{eqnarray}
\norm{\Phi_\phi(w)-\Phi_\phi(h)}_{F^s}
&\leq& c\normo{ L \partial_x(\psi^3(\tau)(w^3(\tau)-h^3(\tau)))}_{F^s}\nonumber\\
&\leq&c(\norm{w}_{F^s}^2+\norm{h}_{F^s}^2)\norm{w-h}_{F^s}\nonumber\\
&\leq&8c^3r^2\norm{w-h}_{F^s}\leq \frac{1}{2}\norm{w-h}_{F^s}.
\end{eqnarray}
Thus $\Phi_\phi(\cdot)$ is a contraction. Therefore, there exists a
unique $u\in \mathcal{B}$ such that
\begin{eqnarray*}
u=\psi(t)W(t)\phi+\psi(t)\int_0^t
W(t-t')(\partial_x[(\psi(t')u)^3](t'))dt'.
\end{eqnarray*}
Hence $u$ solves the integral equation \eqref{eq:inteq} in the time
interval $[-1,1]$.

Part (c) of Theorem \ref{t11} follows from the scaling
\eqref{eq:scaling}, Lemma \ref{l34} and Proposition \ref{p63}. We
prove now part (b). For the real-valued case, according to Theorem
1.2 in \cite{KK}, it suffices to prove that if $s> 1$ then
\[\partial_x u \in L^4_{t\in [0,T]}L^\infty_x.\] Indeed, this follows
from the fact that $u \in F^s(T)$ and $(4,\infty)$ is an admissible
pair and Lemma \ref{l34}. For the complex-valued case, we have some
weak uniqueness. From the proof we see $u$ is unique in the
following set
\begin{eqnarray}\label{eq:Bu0}
B(u_0)=\left\{u:
\begin{array}{l}
u=\lambda^{1/2}\widetilde{u}(\lambda^2 t,\lambda x) \mbox{ in }
[-T,T] \mbox{ for some
} \lambda=\lambda(\norm{\phi}_{H^s})\gg 1\\
 \mbox{and a solution $\widetilde{u}$ to \eqref{eq:truninteq} satisfying
} \norm{\widetilde{u}}_{F^s}\les \lambda^{-s}
\end{array}
\right\}.
\end{eqnarray}
Therefore, we complete the proof of Theorem \ref{t11}.

\begin{remark}
For the real-valued case, the uniqueness actually holds in $F^s(T)$
by the uniqueness in \cite{KenigT} and Lemma \ref{l34}. From the
proof we see the $L^2$ norm smallness condition is due to the
$high\times low\rightarrow high$ interaction where both low
frequency are around $0$. It is also due to this interaction that
one can not apply the methods as in the second part. This bad
interaction is removed via gauge transformation in the previous
results. On the other hand, one may also remove the smallness
condition by performing a gauge transformation as following
\[v(x,t)=e^{-(i/2)\int_{-\infty}^x (P_{\les 1} u(y,t))^2 dy}P_{+}P_{\gg 1}u(x,t),\]
and using the similar methods in \cite{IK}.

\end{remark}

\section{Short-time Trilinear Estimates}

In this section we devote to prove some dyadic trilinear estimates
in the spaces $F_k,\ N_k$. For $k\in \Z$ and $j\in \Z$ let
\begin{eqnarray}
\wt{D}_{k,j}=\{(\xi,\tau): \xi \in I_k, |\tau-\omega(\xi)|\leq
2^j\}.
\end{eqnarray}
Note that $\dot{D}_{k,j}\subset \wt{D}_{k,j}$ for any $k,j\in \Z$.

\begin{proposition}\label{p71} If $k_4\geq 20$, $|k_3-k_4|\leq 5$, $k_1\leq k_2\leq
k_3-10$, then
\begin{eqnarray}\label{eq:p711}
\norm{R_{k_4}\partial_x(u_{k_1}v_{k_2}w_{k_3})}_{N_{k_4}}\leq C
\min(2^{k_1/2}, \jb{k_2})
\norm{u_{k_1}}_{F_{k_1}}\norm{v_{k_2}}_{F_{k_2}}\norm{w_{k_3}}_{F_{k_3}}.
\end{eqnarray}
\end{proposition}
\begin{proof}
Using the definitions and \eqref{eq:pBk3}, we get that the left-hand
side of \eqref{eq:p711} is dominated by
\begin{eqnarray}\label{eq:p712}
&&C\sup_{t_k\in \R}\norm{(\tau-\omega(\xi)+i2^{k_4})^{-1}\cdot
2^{k_4}1_{I_k}(\xi)\cdot \ft[u_{k_1}\eta_0(2^{k_4-2}(t-t_k))] \nonumber\\
&&\quad *\ft[v_{k_2} \eta_0(2^{k_4-2}(t-t_k))]*\ft[w_{k_3}
\eta_0(2^{k_4-2}(t-t_k))]}_{B_k}.
\end{eqnarray}
It suffices to prove that if $j_i\geq k_4$ and $f_{k_i,j_i}:
\R^2\rightarrow \R_+$ are supported in $\widetilde{D}_{k_i,j_i}$ for
$i=1,2,3$, then
\begin{eqnarray}\label{eq:p713}
&&2^{k_4}\sum_{j_4\geq
k_4}2^{-j_4/2}\norm{1_{\widetilde{D}_{k_4,j_4}}\cdot
(f_{k_1,j_1}*f_{k_2,j_2}*f_{k_3,j_3})}_{L^2}\nonumber\\
&\leq& C\min(2^{k_1/2},
\jb{k_2})2^{j_1/2}\norm{f_{k_1,j_1}}_{L^2}2^{j_2/2}\norm{f_{k_2,j_2}}_{L^2}2^{j_3/2}\norm{f_{k_3,j_3}}_{L^2}.
\end{eqnarray}

We assume first \eqref{eq:p713}. Let $f_{k_1}=\ft[u_{k_1}\cdot
\eta_0(2^{k_4-2}(t-t_k))]$, $f_{k_2}=\ft[v_{k_2}\cdot
\eta_0(2^{k_4-2}(t-t_k))]$ and $f_{k_3}=\ft[w_{k_3}\cdot
\eta_0(2^{k_4-2}(t-t_k))]$. Then from the definition of $B_k$ we get
that \eqref{eq:p712} is dominated by
\begin{eqnarray}\label{eq:p714}
\sup_{t_k\in
\R}2^{k_4}\sum_{j_4=0}^{\infty}2^{j_4/2}\sum_{j_1,j_2,j_3\geq
k_4}\norm{(2^{j_4}+i2^{k_4})^{-1}1_{\dot{D}_{k_4,j_4}}\cdot
f_{k_1,j_1}*f_{k_2,j_2}*f_{k_3,j_3}}_{L^2},
\end{eqnarray}
where $f_{k_i,j_i}=f_{k_i}(\xi,\tau)\eta_{j_i}(\tau-\omega(\xi))$
for $j_i>k_4$ and $f_{k_i,k_4}=f_{k_i}(\xi,\tau)\eta_{\leq
k_4}(\tau-\omega(\xi))$, $i=1,2,3$. For the summation on the terms
$j_4<k_4$ in \eqref{eq:p714}, we get from the fact
$1_{D_{k_4,j_4}}\leq 1_{\wt{D}_{k_4,j_4}}$ that
\begin{eqnarray}
&&\sup_{t_k\in
\R}2^{k_4}\sum_{j_4<k_4}2^{j_4/2}\sum_{j_1,j_2,j_3\geq
k_4}\norm{(2^{j_4}+i2^{k_4})^{-1}1_{\dot{D}_{k_4,j_4}}\cdot
f_{k_1,j_1}*f_{k_2,j_2}*f_{k_3,j_3}}_{L^2}\nonumber\\
&&\les \sup_{t_k\in \R}\sum_{j_1,j_2,j_3\geq
k_4}2^{k_4/2}\norm{1_{\wt{D}_{k_4,k_4}}\cdot
f_{k_1,j_1}*f_{k_2,j_2}*f_{k_3,j_3}}_{L^2}.
\end{eqnarray}

From the fact that $f_{k_i,j_i}$ is supported in $\wt{D}_{k_i,j_i}$
for $i=1,2,3$ and using \eqref{eq:p713}, we get that
\begin{eqnarray*}
&&\sup_{t_k\in \R}\sum_{j_1,j_2,j_3\geq k_4}2^{k_4}\sum_{j_4\geq
k_4}2^{-j_4/2}\norm{1_{\wt{D}_{k_4,j_4}}\cdot
f_{k_1,j_1}*f_{k_2,j_2}*f_{k_3,j_3}}_{L^2}\\
&&\les\sup_{t_k\in \R} \min(2^{k_1/2},
\jb{k_2})\sum_{j_1,j_2,j_3\geq
k_4}2^{j_1/2}\norm{f_{k_1,j_1}}_{L^2}2^{j_2/2}\norm{f_{k_2,j_2}}_{L^2}2^{j_3/2}\norm{f_{k_3,j_3}}_{L^2}.
\end{eqnarray*}
Thus using \eqref{eq:pBk2} and \eqref{eq:pBk3} we obtain
\eqref{eq:p711}, as desired.

To prove \eqref{eq:p713}, we consider first the case $|k_1-k_2|\leq
5$. If $k_2\geq 0$, it follows from Corollary \ref{cor42} (b) and
Remark \ref{rem43} that
\begin{eqnarray}\label{eq:p715}
&&2^{k_4}\sum_{j_4\geq
k_4}2^{-j_4/2}\norm{1_{\widetilde{D}_{k_4,j_4}}\cdot
(f_{k_1,j_1}*f_{k_2,j_2}*f_{k_3,j_3})}_{L^2}\nonumber\\
&\les& 2^{k_4}\sum_{j_4\geq
k_4+k_2}2^{-j_4/2}2^{(j_1+j_2+j_3)/2}2^{-k_4/2}2^{k_1/2}\prod_{i=1}^3\norm{f_{k_i,j_i}}_{L^2}\nonumber\\
&& +2^{k_4}\sum_{k_4\leq j_4\leq
k_4+k_2}2^{-j_4/2}2^{(j_1+j_2+j_3+j_4)/2}2^{-k_4}\prod_{i=1}^3\norm{f_{k_i,j_i}}_{L^2}\nonumber\\
&\les&
(1+k_2)2^{(j_1+j_2+j_3)/2}\prod_{i=1}^3\norm{f_{k_i,j_i}}_{L^2},
\end{eqnarray}
which is \eqref{eq:p713} as desired. If $k_2<0$, then from Corollary
\ref{cor42} (a) and Remark \ref{rem43} we get that
\begin{eqnarray}\label{eq:p716}
&&2^{k_4}\sum_{j_4\geq
k_4}2^{-j_4/2}\norm{1_{\widetilde{D}_{k_4,j_4}}\cdot
(f_{k_1,j_1}*f_{k_2,j_2}*f_{k_3,j_3})}_{L^2}\nonumber\\
&\les&2^{k_4}\sum_{j_4\geq
k_4}2^{-j_4/2}2^{(j_1+j_2+j_3)/2}2^{-j_3/2}2^{k_1/2}2^{k_2/2}\prod_{i=1}^3\norm{f_{k_i,j_i}}_{L^2}\nonumber\\
&\les&
2^{k_1/2}2^{(j_1+j_2+j_3)/2}\prod_{i=1}^3\norm{f_{k_i,j_i}}_{L^2},
\end{eqnarray}
which is \eqref{eq:p713} as desired. We assume now $k_1<k_2-5$. If
$k_2<0$, then arguing as in \eqref{eq:p716} we get \eqref{eq:p713}
as desired. If $k_2>0$, then \eqref{eq:p715} also holds in this
case. On the other hand, by checking the support properties of the
function $f_{k_i,j_i}$, $i=1,2,3$, we get that
$1_{\widetilde{D}_{k_4,j_4}}\cdot
(f_{k_1,j_1}*f_{k_2,j_2}*f_{k_3,j_3})\equiv 0$ unless $j_{max}\geq
k_4+k_2-20$. For the summation on the terms $j_4>k_4+k_2-30$ in
\eqref{eq:p713}, we have
\begin{eqnarray}
&&2^{k_4}\sum_{j_4\geq
k_4+k_2-30}2^{-j_4/2}\norm{1_{\widetilde{D}_{k_4,j_4}}\cdot
(f_{k_1,j_1}*f_{k_2,j_2}*f_{k_3,j_3})}_{L^2}\nonumber\\
&\les&2^{k_4}\sum_{j_4\geq
k_4+k_2-30}2^{-j_4/2}2^{(j_1+j_2+j_3)/2}2^{-j_3/2}2^{k_1/2}2^{k_2/2}\prod_{i=1}^3\norm{f_{k_i,j_i}}_{L^2}\nonumber\\
&\les&
2^{k_1/2}2^{(j_1+j_2+j_3)/2}\prod_{i=1}^3\norm{f_{k_i,j_i}}_{L^2}.
\end{eqnarray}
For the summation on the terms $j_4<k_4+k_2-30$, we have $j_4\leq
j_{med}$. Thus using Corollary \ref{cor42} (a), then we get
\begin{eqnarray}
&&2^{k_4}\sum_{j_4<
k_4+k_2-30}2^{-j_4/2}\norm{1_{\widetilde{D}_{k_4,j_4}}\cdot
(f_{k_1,j_1}*f_{k_2,j_2}*f_{k_3,j_3})}_{L^2}\nonumber\\
&\les&2^{k_4}\sum_{k_4\leq j_4<
k_4+k_2-30}2^{-j_4/2}2^{(j_1+j_2+j_3)/2}2^{-j_{max}/2}2^{k_1/2}2^{k_2/2}\prod_{i=1}^3\norm{f_{k_i,j_i}}_{L^2}\nonumber\\
&\les&
2^{k_1/2}2^{(j_1+j_2+j_3)/2}\prod_{i=1}^3\norm{f_{k_i,j_i}}_{L^2}.
\end{eqnarray}
Therefore, we complete the proof of the proposition.
\end{proof}

\begin{proposition}\label{p72}I
f $k_4\geq 20$, $|k_3-k_4|\leq 5$, $k_3-10\leq k_2\leq
k_3$ and $k_1\leq k_2-20$, then we have
\begin{eqnarray}
\norm{R_{k_4}\partial_x(u_{k_1}v_{k_2}w_{k_3})}_{N_{k_4}}\leq
C\min(2^{k_1/2}, 1)
\norm{u_{k_1}}_{F_{k_1}}\norm{v_{k_2}}_{F_{k_2}}\norm{w_{k_3}}_{F_{k_3}}.
\end{eqnarray}
\end{proposition}
\begin{proof}
As in the proof of Proposition \ref{p71}, using \eqref{eq:pBk2} and
\eqref{eq:pBk3}, we see that it suffices to prove that if
$j_1,j_2,j_3\geq k_4$, and $f_{k_i,j_i}: \R^2\rightarrow \R_+$ are
supported in $\widetilde{D}_{k_i,j_i}$, $i=1,2,3$, then
\begin{eqnarray}\label{eq:p722}
&&2^{k_4}\sum_{j_4\geq
k_4}2^{-j_4/2}\norm{1_{\widetilde{D}_{k_4,j_4}}\cdot
(f_{k_1,j_1}*f_{k_2,j_2}*f_{k_3,j_3})}_{L^2}\nonumber\\
&&\leq C\min(2^{k_1/2},
1)\cdot2^{j_1/2}\norm{f_{k_1,j_1}}_{L^2}\cdot2^{j_2/2}\norm{f_{k_2,j_2}}_{L^2}\cdot2^{j_3/2}\norm{f_{k_3,j_3}}_{L^2}.\qquad
\end{eqnarray}
Since in the area $\{|\xi_i| \in \widetilde{I}_{k_i},i=1,2,3\}\cap
\{|\xi_1+\xi_2+\xi_3|\in I_{k_4}\}$
\[|\Omega(\xi_1,\xi_2,\xi_3)|\sim 2^{2k_4},\]
then by checking the support properties, we get
$1_{\widetilde{D}_{k_4,j_4}}\cdot
(f_{k_1,j_1}*f_{k_2,j_2}*f_{k_3,j_3})\equiv 0$ unless $j_{max}\geq
2k_4-30$. It follows from Corollary \ref{cor42} (a) that the
left-hand side of \eqref{eq:p722} is bounded by
\begin{eqnarray}
2^{k_4}\sum_{j_4\geq
k_4}2^{-j_4/2}2^{(j_1+j_2+j_3+j_4)/2}2^{-j_{max}/2}2^{-j_{sub}/2}2^{k_1/2}2^{k_2/2}\prod_{i=1}^3\norm{f_{k_i,j_i}}_{L^2}.
\end{eqnarray}
Then we get the bound \eqref{eq:p722} by considering either
$j_4=j_{max}$ or $j_4\ne j_{max}$.
\end{proof}

\begin{proposition}
If $k_4\geq 20$, $|k_3-k_4|\leq 5$, $k_3-10\leq k_2\leq k_3$ and
$k_2-30\leq k_1\leq k_2$, then we have
\begin{eqnarray}
\norm{R_{k_4}\partial_x(u_{k_1}v_{k_2}w_{k_3})}_{N_{k_4}}\leq C
2^{k_3/2}
\norm{u_{k_1}}_{F_{k_1}}\norm{v_{k_2}}_{F_{k_2}}\norm{w_{k_3}}_{F_{k_3}}.
\end{eqnarray}
\end{proposition}
\begin{proof}
As in the proof of Proposition \ref{p71}, using \eqref{eq:pBk2} and
\eqref{eq:pBk3}, it suffices to prove that if $j_1,j_2,j_3\geq k_4$,
and $f_{k_i,j_i}: \R^2\rightarrow \R_+$ are supported in
$\widetilde{D}_{k_i,j_i}$, $i=1,2,3$, then
\begin{eqnarray*}
&&2^{k_4}\sum_{j_4\geq
k_4}2^{-j_4/2}\norm{1_{\widetilde{D}_{k_4,j_4}}\cdot
(f_{k_1,j_1}*f_{k_2,j_2}*f_{k_3,j_3})}_{L^2}\\
&&\leq
C2^{k_4/2}\cdot2^{j_1/2}\norm{f_{k_1,j_1}}_{L^2}\cdot2^{j_2/2}\norm{f_{k_2,j_2}}_{L^2}\cdot2^{j_3/2}\norm{f_{k_3,j_3}}_{L^2},
\end{eqnarray*}
which follows immediately from Corollary \ref{cor42} (c).
\end{proof}

\begin{proposition}\label{p74}
If $k_3\geq 20$, $|k_3-k_2|\leq 5$, $k_2-10\leq k_1\leq k_2$ and $
k_4\leq k_1-30$, then we have
\begin{eqnarray}\label{eq:p741}
\norm{R_{k_4}\partial_x(u_{k_1}v_{k_2}w_{k_3})}_{N_{k_4}}\leq
C\min(2^{k_4},1 )|k_2|
\norm{u_{k_1}}_{F_{k_1}}\norm{v_{k_2}}_{F_{k_2}}\norm{w_{k_3}}_{F_{k_3}}.
\end{eqnarray}
\end{proposition}
\begin{proof}
Let $\gamma:\R\rightarrow [0,1]$ denote a smooth function supported
in $[-1,1]$ with the property that
\[\sum_{n\in \Z}\gamma^3(x-n)\equiv 1, \quad x\in \R.\]
Using the definitions, the left-hand side of \eqref{eq:p741} is
dominated by
\begin{eqnarray*}
&&C\sup_{t_k\in \R}\normb{(\tau-\omega(\xi)+i2^{{k_4}_+})^{-1}\cdot
2^{k_4}1_{I_{k_4}}(\xi)\cdot \sum_{|m|\leq C2^{k_2-{k_4}_+}}\\
&&\quad \ft[u_{k_1}\eta_0(2^{{k_4}_+}(t-t_k))\gamma(2^{{k_2}}(t-t_k)-m)]*\\
&&\quad \ft[v_{k_2}\eta_0(2^{{k_4}_+}(t-t_k))\gamma(2^{{k_2}}(t-t_k)-m)]*\\
&&\quad
\ft[w_{k_3}\eta_0(2^{{k_4}_+}(t-t_k))\gamma(2^{{k_2}}(t-t_k)-m)]}_{B_k}.
\end{eqnarray*}
In view of the definitions, \eqref{eq:pBk2} and \eqref{eq:pBk3}, it
suffices to prove that if $j_1,j_2,j_3\geq k_2$, and
$f_{k_i,j_i}:\R^3 \ra \R_+$ are supported in
$\widetilde{D}_{k_i,j_i}$, $i=1,2,3$, then
\begin{eqnarray}\label{eq:p742}
&&2^{k_4}2^{k_2-{k_4}_+}\sum_{j_4\geq
{k_4}_+}2^{-j_4/2}\norm{1_{\widetilde{D}_{k_4,j_4}}\cdot
(f_{k_1,j_1}*f_{k_2,j_2}*f_{k_3,j_3})}_{L^2}\nonumber\\
&&\leq C\min(2^{k_4},1
)|k_2|\cdot2^{j_1/2}\norm{f_{k_1,j_1}}_{L^2}\cdot2^{j_2/2}\norm{f_{k_2,j_2}}_{L^2}\cdot2^{j_3/2}\norm{f_{k_3,j_3}}_{L^2}.\qquad
\end{eqnarray}
From the same argument as in Proposition \ref{p72}, we get
$j_{max}\geq 2k_2-30$. Then \eqref{eq:p742} follows from Corollary
\ref{cor42}.
\end{proof}

\begin{proposition}
If $k_3\geq 20$, $|k_3-k_2|\leq 5$, and $k_1, k_4\leq k_2-10$, then
\begin{eqnarray}
\norm{R_{k_4}\partial_x(u_{k_1}v_{k_2}w_{k_3})}_{N_{k_4}}\leq
C\min(2^{k_1/2},1 )|k_2|
\norm{u_{k_1}}_{F_{k_1}}\norm{v_{k_2}}_{F_{k_2}}\norm{w_{k_3}}_{F_{k_3}}.
\end{eqnarray}
\end{proposition}
\begin{proof}
As in the proof of Proposition \ref{p74}, it suffices to prove that
if $j_1,j_2,j_3\geq k_2$, and $f_{k_i,j_i}:\R^3 \ra \R_+$ are
supported in $\widetilde{D}_{k_i,j_i}$, $i=1,2,3$, then
\begin{eqnarray}\label{eq:p752}
&&2^{k_4}2^{k_2-{k_4}_+}\sum_{j_4\geq
{k_4}_+}2^{-j_4/2}\norm{1_{\widetilde{D}_{k_4,j_4}}\cdot
(f_{k_1,j_1}*f_{k_2,j_2}*f_{k_3,j_3})}_{L^2}\nonumber\\
&&\leq C\min(2^{k_1/2},1
)|k_2|\cdot2^{j_1/2}\norm{f_{k_1,j_1}}_{L^2}\cdot2^{j_2/2}\norm{f_{k_2,j_2}}_{L^2}\cdot2^{j_3/2}\norm{f_{k_3,j_3}}_{L^2}.\qquad
\quad
\end{eqnarray}
In proving \eqref{eq:p752} we may assume $j_4\leq 10 k_2$ in the
summation of \eqref{eq:p752}, otherwise we use Corollary \ref{cor42}
(a). Using Corollary \ref{cor42} (a) for $k_1\leq 0$, else using
Corollary \ref{cor42} (b), then we get
\begin{eqnarray*}
&&2^{k_4}2^{k_2-{k_4}_+}\sum_{j_4\geq
{k_4}_+}2^{-j_4/2}\norm{1_{\widetilde{D}_{k_4,j_4}}\cdot
(f_{k_1,j_1}*f_{k_2,j_2}*f_{k_3,j_3})}_{L^2}\\
&&\leq C\min(2^{k_1/2},1
)|k_2|2^{j_1/2}\norm{f_{k_1,j_1}}_{L^2}2^{j_2/2}\norm{f_{k_2,j_2}}_{L^2}2^{j_3/2}\norm{f_{k_3,j_3}}_{L^2}.
\end{eqnarray*}
Therefore, we complete the proof of the proposition.
\end{proof}

\begin{proposition}\label{p76}
If $k_1,k_2,k_3,k_4\leq 200$, then
\begin{eqnarray}
\norm{R_{k_4}\partial_x(u_{k_1}v_{k_2}w_{k_3})}_{N_{k_4}}\leq
C2^{k_{min}/2}2^{k_{thd}/2}
\norm{u_{k_1}}_{F_{k_1}}\norm{v_{k_2}}_{F_{k_2}}\norm{w_{k_3}}_{F_{k_3}}.
\end{eqnarray}
\end{proposition}
\begin{proof}
This follows immediately from the definitions, Corollary \ref{cor42}
(a), Remark \ref{rem43} and \eqref{eq:pBk2} and \eqref{eq:pBk3}.
\end{proof}

\section{Energy Estimates}

In this section we prove an energy estimate by using I-method
\cite{CKSTT}, following some ideas in \cite{KochTataru}. For the
difference equation of two modified Benjamin-Ono equations, we don't
know how to prove a similar energy estimate due to the lack of
symmetry. That's why we can only solve the half problem of
Conjecture.

\begin{proposition}\label{energyes}
Assume that $T\in (0,1]$ and $u\in C([-T,T]:H^\infty)$ is a
real-valued solution of the initial value problem
\begin{eqnarray}\label{eq:mBOenergy}
\left \{
\begin{array}{l}
u_t+\Hl u_{xx}=u^2u_x,\ (x,t)\in \R\times (-T,T);\\
u(x,0)=\phi(x),
\end{array}
\right.
\end{eqnarray}
Then, for $0\leq l<1/4$ and $s>1/4$, there exists $\delta_0>0$ such
that if $\norm{u}_{E^{l,s}(T)}\leq \delta_0$ then we have
\begin{equation}
\norm{u}_{E^{l,s}(T)}^2\les \norm{\phi}_{\dot{H}^l\cap
\dot{H}^s}^2+\norm{u}_{F^{\rev 4-,\rev 4+}(T)}^4\cdot
\norm{u}_{F^{l,s}(T)}^2.
\end{equation}
\end{proposition}

The following definition was first introduced in \cite{KochTataru}.
\begin{definition}
Let $s\in \R$ and $\epsilon>0$. Then $S_\epsilon^s$ is the class of
spherical symmetric symbols with the following properties:

(i) symbol regularity,
\[|\partial^\alpha a(\xi)|\les a(\xi)(1+\xi^2)^{-\alpha/2}\]

(ii) decay at infinity, for $|\xi|\gg 1$,
\[s\leq \frac{\log a(\xi)}{\log (1+\xi^2)}\leq s+\epsilon, \quad s-\epsilon \leq \frac{d\log a(\xi)}{d\log (1+\xi^2)}\leq s+\epsilon.\]
\end{definition}

Assume $u\in C([-T,T]:H^\infty)$ solves \eqref{eq:mBOenergy} and
$a\in S_\epsilon^s$. Denote $A(D)=\ft^{-1}a(\xi)\ft$. We first set
\[E_0(u)=(A(D)u,u)=\int_{\xi_1+\xi_2=0}a(\xi_1)\widehat{u}(\xi_1)\widehat{u}(\xi_2).\]
Using the equation \eqref{eq:mBOenergy} and noting that $a(\xi)$ is
even while $\omega(\xi)$ is odd, then we easily get that
\begin{eqnarray*}
&&\frac{d}{dt}E_0(u)=R_4(u)\\
&&=-\frac{1}{6}\int_{\Gamma_4}i[\xi_1a(\xi_1)+\xi_2a(\xi_2)+\xi_3a(\xi_3)+\xi_4a(\xi_4)]\widehat{u}(\xi_1)\widehat{u}(\xi_2)\widehat{u}(\xi_3)\widehat{u}(\xi_4),
\end{eqnarray*}
where for $k\in \N$, we denote
\[\Gamma_k=\{\xi_1+\xi_2+\ldots+\xi_k=0\}.\]
Following the idea of I-method, we define a multi-linear correction
term to achieve a cancelation
\[E_1(u)=\int_{\Gamma_4}b_4(\xi_1,\xi_2,\xi_3,\xi_4)\widehat{u}(\xi_1)\widehat{u}(\xi_2)\widehat{u}(\xi_3)\widehat{u}(\xi_4),\]
where $b_4$ will be determined soon. Again using the equation
\eqref{eq:mBO}, we get
\begin{eqnarray*}
&&\frac{d}{dt}E_1(u)=R_6(u)\\
&&+\int_{\Gamma_4}ib_4(\xi_1,\xi_2,\xi_3,\xi_4)[\omega(\xi_1)+\omega(\xi_2)+\omega(\xi_3)+\omega(\xi_4)]\widehat{u}(\xi_1)\widehat{u}(\xi_2)\widehat{u}(\xi_3)\widehat{u}(\xi_4),
\end{eqnarray*}
where
\[R_6(u)=C\int_{\Gamma_6} b_4(\xi_1,\xi_2,\xi_3,\xi_4+\xi_5+\xi_6)(\xi_4+\xi_5+\xi_6)\prod_{j=1}^6 \widehat{u}(\xi_j).\]
To achieve the cancelation of the quadrilinear form we define $b_4$
on $\Gamma_4$ by
\[b_4(\xi_1,\xi_2,\xi_3,\xi_4)=C\frac{\xi_1a(\xi_1)+\xi_2a(\xi_2)+\xi_3a(\xi_3)+\xi_4a(\xi_4)}{\omega(\xi_1)+\omega(\xi_2)+\omega(\xi_3)+\omega(\xi_4)}.\]
Thus we get
\[\frac{d}{dt}(E_0(u)+E_1(u))=R_6(u).\]

\begin{proposition}\label{exb4}
Assume that $a\in S_\epsilon ^s$. Then for each dyadic $\lambda\leq
\alpha \leq \mu$ there is an extension of $b_4$ from the diagonal
set
\[\{(\xi_1,\xi_2,\xi_3,\xi_4)\in \Gamma_4, |\xi_1|\sim \lambda, |\xi_2|\sim \alpha, |\xi_3|, |\xi_4|\sim \mu\}\]
to the full dyadic set
\[\{(\xi_1,\xi_2,\xi_3,\xi_4)\in \R_4, |\xi_1|\sim \lambda, |\xi_2|\sim \alpha, |\xi_3|, |\xi_4|\sim \mu\}\]
which satisfies
\begin{equation}\label{eq:b4bd}
|b_4(\xi_1,\xi_2,\xi_3,\xi_4)|\les a(\mu)\mu^{-1}
\end{equation}
and
\begin{equation}\label{eq:b4dbd}
\sum_{j=1}^4|\partial_j b_4(\xi_1,\xi_2,\xi_3,\xi_4)|\les
a(\alpha)\mu^{-1}+ a(\mu)\mu^{-2}.
\end{equation}
\end{proposition}
\begin{proof}
From symmetry we may assume $|\xi_1|\leq |\xi_2|\leq |\xi_3|\leq
|\xi_4|$ and $\xi_3>0, \xi_4<0$. We first consider the case that
$\xi_1\xi_2>0$, say $\xi_1,\xi_2>0$. Then
$\omega(\xi_1)+\omega(\xi_2)+\omega(\xi_3)+\omega(\xi_4)=-2(\xi_1\xi_2+\xi_2\xi_3+\xi_1\xi_3)$.
Thus in $\Gamma_4$ we have
\begin{eqnarray*}
Cb_4(\xi_1,\xi_2,\xi_3,\xi_4)=\frac{\xi_1a(\xi_1)+\xi_2a(\xi_2)}{\xi_1\xi_2+(\xi_2+\xi_1)\xi_3}+\frac{\xi_3a(\xi_3)+\xi_4a(\xi_4)}{\xi_1\xi_2+\xi_2\xi_3+\xi_1\xi_3}.
\end{eqnarray*}
Using $\xi_1+\xi_2+\xi_3+\xi_4=0$ we get
\begin{eqnarray*}
&&\frac{\xi_3a(\xi_3)+\xi_4a(\xi_4)}{\xi_1\xi_2+\xi_2\xi_3+\xi_1\xi_3}\\
&&=\frac{\xi_1\xi_2}{\xi_1\xi_2+\xi_2\xi_3+\xi_1\xi_3}\frac{\xi_3a(\xi_3)+\xi_4a(\xi_4)}{\xi_3(\xi_3+\xi_4)}-\frac{\xi_3a(\xi_3)+\xi_4a(\xi_4)}{\xi_3(\xi_3+\xi_4)}.
\end{eqnarray*}
Therefore, we extend $b_4$ by setting
\begin{eqnarray}\label{eq:b4c1}
&&Cb_4(\xi_1,\xi_2,\xi_3,\xi_4)=\frac{\xi_1a(\xi_1)+\xi_2a(\xi_2)}{\xi_1\xi_2+\xi_2\xi_3+\xi_1\xi_3}\nonumber\\
&&+\frac{\xi_1\xi_2}{\xi_1\xi_2+\xi_2\xi_3+\xi_1\xi_3}\frac{\xi_3a(\xi_3)+\xi_4a(\xi_4)}{\xi_3(\xi_3+\xi_4)}-\frac{\xi_3a(\xi_3)+\xi_4a(\xi_4)}{\xi_3(\xi_3+\xi_4)}.
\end{eqnarray}
It is easy to see from the properties of $a(\xi)$ that
\[|b_4(\xi_1,\xi_2,\xi_3,\xi_4)|\les a(\mu)\mu^{-1}.\]
It remains to check the derivatives. We only consider $|\partial_1
b_4|$, since the others can be handled in the similar ways. For
$|\partial_1 b_4|$ it suffices to consider the first term on the
right-hand side of \eqref{eq:b4c1}. Direct computations show that
\[\partial_{\xi_1}\brk{\frac{\xi_1a(\xi_1)+\xi_2a(\xi_2)}{\xi_1\xi_2+(\xi_2+\xi_1)\xi_3}}=\frac{[a(\xi_1)-a(\xi_2)]\xi_2\xi_3-\xi_2^2a(\xi_2)}{(\xi_1\xi_2+(\xi_2+\xi_1)\xi_3)^2}+\frac{a'(\xi_1)\xi_1}{\xi_1\xi_2+(\xi_2+\xi_1)\xi_3},\]
which satisfies \eqref{eq:b4dbd} as desired.

We consider now $\xi_1\xi_2<0$, say $\xi_1<0,\xi_2>0$. Thus we get
$\omega(\xi_1)+\omega(\xi_2)+\omega(\xi_3)+\omega(\xi_4)=\xi_1^2-\xi_2^2-\xi_3^2+\xi_4^2=(\xi_1+\xi_2)(\xi_1+\xi_3)$.
We will extend $b_4$ in the following cases.

(a) $\lambda\ll \mu$, $\alpha\leq \mu$. Then the extension of $b_4$
is defined using the formula
\[b_4(\xi_1,\xi_2,\xi_3,\xi_4)=\frac{\xi_1a(\xi_1)+\xi_2a(\xi_2)}{(\xi_1+\xi_2)(\xi_1+\xi_3)}-\frac{\xi_3a(\xi_3)+\xi_4a(\xi_4)}{(\xi_3+\xi_4)(\xi_1+\xi_3)}.\]
Since $\lambda \ll \mu$, we see that $|\xi_1+\xi_3|\sim \mu$. By
using the properties of $a(\xi)$ we see \eqref{eq:b4bd} and
\eqref{eq:b4dbd} are satisfied as desired.

(b) $\lambda\sim \mu$. Then the extension of $b_4$ is defined using
the formula
\[b_4(\xi_1,\xi_2,\xi_3,\xi_4)=\frac{\xi_1a(\xi_1)+\xi_2a(\xi_2)+\xi_3a(\xi_3)-(\xi_1+\xi_2+\xi_3)a(\xi_1+\xi_2+\xi_3)}{(\xi_1+\xi_2)(\xi_1+\xi_3)}.\]
To check the properties, setting
\[q(\xi_1,\xi_2)=\frac{\xi_1a(\xi_1)+\xi_2a(\xi_2)}{\xi_1+\xi_2},\]
then we get that
\[b_4(\xi_1,\xi_2,\xi_3,\xi_4)=\frac{q(\xi_1,\xi_2)-q(\xi_1-(\xi_1+\xi_3),\xi_2+(\xi_1+\xi_3))}{\xi_1+\xi_3},\]
from which we easily verify  \eqref{eq:b4bd} and \eqref{eq:b4dbd}.
\end{proof}

In view of the definition, for Proposition \ref{energyes} we are
mainly concerned with the control of the energy in high frequency.
From the definition we see that if $a\in S_\epsilon^s$ then
$a(\xi)\eta_{\geq 1}(\xi)\in S_\epsilon^s$. Let
\begin{equation}
A_\epsilon^s=\{a(\xi)\eta_{\geq 1}(\xi): a\in S_\epsilon^s\}.
\end{equation}
\begin{proposition}\label{diffenergy}
Assume $a\in A_\epsilon^s$ and $s-\epsilon\geq 0$, then we have
\[|E_1(u)|\les \norm{u}_{\dot{H}^{1/4-}\cap \dot{H}^{1/4+}}^2E_0(u).\]
\end{proposition}
\begin{proof}
Using the definition, we get
\[|E_1(u)|\leq \int_{\Gamma_4} |b_4(\xi_1,\xi_2,\xi_3,\xi_4)|\cdot|\widehat{u}(\xi_1)\widehat{u}(\xi_2)\widehat{u}(\xi_3)\widehat{u}(\xi_4)|.\]
From symmetries, we may assume that $|\xi_1|\leq |\xi_2|\leq
|\xi_3|\leq |\xi_4|$. Localizing $|\xi_j|\sim N_j$ for $N_j$ dyadic
number, we may assume $N_3\sim N_4\ges 1$. Then it follows from
Proposition \ref{exb4} that
\begin{eqnarray*}
|E_1(u)|&\les& \sum_{N_j}\int_{\Gamma_4,|\xi_j|\sim N_j}
{N_4}^{-1}|\widehat{u}(\xi_1)\widehat{u}(\xi_2)a(N_4)^{1/2}\widehat{u}(\xi_3)a(N_4)^{1/2}\widehat{u}(\xi_4)|\\
&\les&\sup_{N_4\ges 1}\sum_{N_1\leq N_2\leq N_4}
N_4^{-1/2+\epsilon}{N_1}^{1/2}\norm{R_{N_1}u}_2\norm{R_{N_2}u}_2E_0(u)\\
&\les&\norm{u}_{\dot{H}^{1/4-}\cap \dot{H}^{1/4+}}^2E_0(u).
\end{eqnarray*}
Therefore we complete the proof of the proposition.
\end{proof}

\begin{proposition}\label{6linear}
Assume $a\in A_\epsilon^s$, $s-\epsilon\geq 0$ and $T\in (0,1]$.
Then
\begin{equation}
\aabs{\int_{-T}^T R_6(u)dt}\les
\norm{u}_{F^{1/4-,1/4+}(T)}^4\norm{u}_{F^{l,s}(T)}^2.
\end{equation}
\end{proposition}
\begin{proof}
We first fix extension $\wt u \in C_0(\R:H^\infty)$ of $u$ such that
$\norm{R_k(\wt u)}_{F_k}\leq 2\norm{R_k(u)}_{F_k(T)}$, $k\in \Z$. It
suffices to prove that
\begin{equation}\label{eq:6linear2}
\aabs{\int_0^T R_6(\wt u)dt}\les \norm{\wt
u}_{F^{1/4-,1/4+}}^4\norm{\wt u}_{F^{l,s}}^2.
\end{equation}
For simplicity of the notations we still write $\wt u=u$. From
symmetry, we get
\begin{eqnarray*}
CR_6(u)&=&\int_{\Gamma_6}
[b_4(\xi_1,\xi_2,\xi_3,\xi_4+\xi_5+\xi_6)(\xi_4+\xi_5+\xi_6)\\
&&-b_4(-\xi_4,-\xi_5,-\xi_6,\xi_4+\xi_5+\xi_6)(\xi_4+\xi_5+\xi_6)]\prod_{j=1}^6
\widehat{u}(\xi_j).
\end{eqnarray*}
Localizing $|\xi_j|\sim N_j=2^{k_j}$ and using symmetry, we may
assume $N_1\leq N_2\leq N_3$, $N_4\leq N_5\leq N_6$ and
$\max(N_j)\sim \sub(N_j)\ges 1$ where $\max(N_j)$ and $\sub(N_j)$
are the maximum and second-maximum of $N_j,j=1,2,\ldots,6$. Let
$u_{k}=R_k(u)$ and $\xi_{456}=\xi_4+\xi_5+\xi_6$. Thus
\begin{eqnarray}\label{eq:6linear3}
\aabs{\int_0^T R_6(u)dt}&\les& \sum_{N_j}\abs{\int_0^T
\int_{\Gamma_6}
[b_4(\xi_1,\xi_2,\xi_3,\xi_{456})(\xi_{456})\nonumber\\
&&-b_4(-\xi_4,-\xi_5,-\xi_6,\xi_{456})(\xi_{456})]\prod_{j=1}^6
\widehat{u_{k_j}}(\xi_j)dt}.
\end{eqnarray}
Let $\gamma:\R\ra [0,1]$ denote a positive smooth function supported
in $[-1,1]$ with the property that
\[\sum_{n\in \Z}\gamma^6(x-n)\equiv 1, \quad x \in \R.\]
We will bound \eqref{eq:6linear2} in several cases.

{\bf Case 1.} $N_3\les N_5, N_6$ and $N_5\sim N_6 \ges 1$. Then we
get that the right-hand side of \eqref{eq:6linear3} is bounded by
\begin{eqnarray}\label{eq:6linear4}
&&\sum_{N_j}\sum_{|n|\les 2^{k_6}}\abs{\int_\R \int_{\Gamma_6}
[b_4(\xi_1,\xi_2,\xi_3,\xi_{456})-b_4(-\xi_4,-\xi_5,-\xi_6,\xi_{456})]\nonumber\\
&&\quad \xi_{456}[\gamma(2^{k_6}t-n)1_{[0,T]}(t)
\widehat{u_{k_1}}(\xi_1)]\prod_{j=2}^6[\gamma(2^{k_6}t-n)
\widehat{u_{k_j}}(\xi_j)]dt}.
\end{eqnarray}
We observe first that
\[|A|=|\{n: \gamma(2^{k_6}t-n)1_{[0,T]}(t)\ne \gamma(2^{k_6}t-n), 0,\ \forall\ t \in \R\}|\leq 4.\]
Let $f_{k_j}(\xi,\tau)=\ft[\gamma(2^{k_6}t-n)R_{k}u]$,
$j=1,2,\ldots,6$. Using proposition \ref{exb4} and Plancherel's
theorem, we easily get that the summation for $n\in A^c$ of
\eqref{eq:6linear4} is bounded by
\begin{eqnarray}\label{eq:6linear5}
\sum_{N_j}\sum_{|n|\les 2^{k_6},n\in
A^c}(a(N_3)N_3^{-1}+a(N_6)N_6^{-1})N_3\aabs{
\int_{\Gamma_6(\R^2)}\prod_{j=1}^6 f_{k_j}(\xi_j,\tau_j)}.
\end{eqnarray}
Using H\"older's inequality and the embedding properties of $B_k$,
we get
\begin{eqnarray*}
\aabs{ \int_{\Gamma_6(\R^2)}\prod_{j=1}^6
f_{k_j}(\xi_j,\tau_j)}&\les&
\prod_{j=1}^4\norm{\ft^{-1}(f_{k_j})}_{L_x^4L_t^\infty}\\
&&\norm{\ft^{-1}(f_{k_5})}_{L_x^\infty
L_t^2}\norm{\ft^{-1}(f_{k_6})}_{L_x^\infty L_t^2}\\
&\les&
N_5^{-1}\prod_{j=1}^4N_j^{1/4}\prod_{j=1}^6\norm{f_{k_j}}_{B_k}.
\end{eqnarray*}
Then we can bound \eqref{eq:6linear5} by
\begin{eqnarray*}
&&\sum_{N_j}(a(N_3)N_3^{-1}+a(N_6)N_6^{-1})N_3\prod_{j=1}^4N_j^{1/4}\prod_{j=1}^6\norm{f_{k_j}}_{F_k}\\
&\les& \norm{u}_{F^{1/4-,1/4+}}^4\norm{u}_{F^{l,s}}^2,
\end{eqnarray*}
which is \eqref{eq:6linear2} as desired.

For the summation of $n\in A$, we observe that if $I\subset \R$ is
an interval, $k\in \Z$, $f_k\in B_k$, and $f_k^I=\ft(1_I(t)\cdot
\ft^{-1}(f_k))$ then
\[\sup_{j\in \Z_+}2^{j/2}\norm{\eta_j(\tau-\omega(\xi))\cdot f_k^I}_{L^2}\les \norm{f_k}_{B_k}.\]
Let
$f_{k_i,j_i}=\eta_{j_i}(\tau-\omega(\xi))\ft[\gamma(2^{k_6}t-n)R_{k}u]$,
$i=1,2,\ldots,5$, and
$f_{k_6,j_6}=\eta_{j_6}(\tau-\omega(\xi))\ft[\gamma(2^{k_6}t-n)1_{[0,T]}(t)R_{k}u]$.
If $j_6\geq 100k_6$, then by checking the support properties we get
$\int_{\Gamma_6(\R^2)}\prod_{i=1}^6 f_{k_i,j_i}(\xi_j,\tau_j)\equiv
0$ unless $|j_{max}-j_{sub}|\leq 10$ and $j_{max}\geq 100k_6$, where
$j_{max}$ and $j_{sub}$ are the maximum and sub-maximum of
$j_1,j_2,\ldots,j_6$. By using Cauchy-Schwarz inequality, we get
that
\begin{eqnarray*}
&&\sum_{j_i}\aabs{\int_{\Gamma_6(\R^2)}\prod_{i=1}^6
f_{k_i,j_i}(\xi_j,\tau_j)}\\
&\les& 2^{(k_1+\ldots +k_6)/2}2^{(j_1+\ldots
+j_6)/2}2^{-(j_{max})}\prod_{i=1}^6\norm{f_{k_i,j_i}}_{L^2}\\
&\les&
N_5^{-1}\prod_{j=1}^4N_j^{1/4}\prod_{j=1}^6\norm{f_{k_j}}_{F_k},
\end{eqnarray*}
which is acceptable. If $j_6\leq 100k_6$ then we argue as before for
$n\in A^c$, hence we get that
\begin{eqnarray*}
\sum_{j_i}\aabs{\int_{\Gamma_6(\R^2)}\prod_{i=1}^6
f_{k_i,j_i}(\xi_j,\tau_j)} \les
k_6N_5^{-1}\prod_{j=1}^4N_j^{1/4}\prod_{j=1}^6\norm{f_{k_j}}_{F_k},
\end{eqnarray*}
which combined with \eqref{eq:6linear5} gives \eqref{eq:6linear2}.

{\bf Case 2.} $N_6\les N_2,N_3$ and $N_2\sim N_3\ges 1$. From
symmetry, this case is identical to Case 1. We omit the details.

{\bf Case 3.} $N_2,N_5\ll N_3,N_6$ and $N_6\sim N_3\ges 1$. From the
proof of Proposition \ref{exb4}, we get that
\[|b_4(\xi_1,\xi_2,\xi_3,\xi_{456})-b_4(-\xi_4,-\xi_5,-\xi_6,\xi_{456})|\cdot|\xi_{456}|\les a(N_3)+a(N_6)N_6^{-1}.\]
Then following the same argument as in Case 1, we obtain
\eqref{eq:6linear2} as desired.
\end{proof}

\begin{lemma}[Lemma 5.5, \cite{KochTataru}]\label{Hsre}
There is a sequence $\{\beta_\lambda\}$ with the following
properties:

(a) $\lambda^{2s}\norm{P_{\lambda}(u_0)}_{L^2}^2\leq \beta_\lambda
\norm{u_0}_{H^s}^2$,

(b) $\sum \beta_\lambda \les 1$,

(c) $\beta_\lambda$ is slowly varying in the sense that
\[|\log_2{\beta_\lambda}-\log_2{\beta_\mu}|\leq \half \epsilon |\log_2\lambda-\log_2\mu|.\]
\end{lemma}

\noindent{\bf Proof of Proposition \ref{energyes}.} \quad In view of
the definition, we get
\begin{eqnarray}
\norm{u}_{E^{l,s}(T)}^2\les \norm{P_{\leq
0}u_0}_{\dot{H}^l}^2+\sum_{k\geq 1}\sup_{t\in
[-T,T]}2^{2ks}\norm{P_k(u(t))}_{L^2}^2.
\end{eqnarray}
We will prove that if $k\geq 1$ then
\begin{eqnarray}\label{eq:energy2}
\sup_{t\in [-T,T]}2^{2ks}\norm{P_k(u(t))}_{L^2}^2\les
\beta_k(\norm{P_{\geq 1}u_0}_{H^s}^2+\norm{u}_{F^{\rev 4-,\rev
4+}(T)}^4\cdot \norm{u}_{F^{l,s}(T)}^2),
\end{eqnarray}
which suffices to prove Proposition \ref{energyes} in view of Lemma
\ref{Hsre} (b). In order to prove \eqref{eq:energy2} for some fixed
$k_0$ we define the sequence
\[a_k=2^{2ks}\max(1,\beta_{k_0}^{-1}2^{-\epsilon|k-k_0|}).\]
Using the slowly varying condition (iii), then we get
\begin{eqnarray*}
\sum_{k\geq 1}a_k\norm{P_k(u_0)}_{L^2}^2&\les&
\sum_{k}2^{2ks}\norm{P_k(u_0)}_{L^2}^2\\
&&+2^{-\epsilon|k-k_0|/2}2^{2ks}\beta_k^{-1}\norm{P_k(u_0)}_{L^2}^2\\
&\les& \norm{P_{\geq 1}(u_0)}_{H^s}^2.
\end{eqnarray*}
We may assume that $\beta_0=1$. Then we see that $\max(|\beta_k|,
|\beta_k^{-1}|)\leq 2^{k\epsilon/2}$. Correspondingly we find a
function $a(\xi)\in S_\epsilon^s$ so that
\[a(\xi)\sim a_k,\quad |\xi|\sim 2^k.\]
Thus we apply Proposition \ref{diffenergy}, \ref{6linear} for
$a(\xi)\eta_{\geq 1}(\xi)$, then we get
\[\sup_{t\in [-T,T]}|E_0(u(t)+E_1(u(t))|\leq |E_0(u_0)+E_1(u_0)|+\left|\int_{-T}^T R_6(u)dt\right|,\]
from which we see that
\[\sup_{t\in [-T,T]}|E_0(u(t)|\leq |E_0(u_0)|+\norm{u}_{F^{1/4-,1/4+}(T)}^4\norm{u}_{F^{l,s}(T)}^2.\]
Therefore, we get
\[\brk{\sum_{k\geq 1}a_k\norm{P_k(u(t))}_{L^2}^2}\les \norm{P_{\geq 1}u_0}_{H^s}^2+\norm{u}_{F^{1/4-,1/4+}(T)}^4\norm{u}_{F^{l,s}(T)}^2,\]
which at $k=k_0$ gives \eqref{eq:energy2} as desired.
\endprf

\section{Proof of Theorem \ref{aprioribound}}

In this section we devote to prove Theorem \ref{aprioribound}. The
main ingredients are energy estimates and short-time trilinear
estimates. The idea is due to Ionescu, Kenig and Tataru \cite{IKT}.

\begin{proposition}\label{p91}
Let $l,s\geq 0$, $T\in (0,1]$, and $u\in F^{l,s}(T)$, then
\begin{equation}
\sup_{t\in [-T,T]}\norm{u(t)}_{\dot{H}^l\cap \dot{H}^s}\les\
\norm{u}_{F^{l,s}(T)}.
\end{equation}
\end{proposition}
\begin{proof}
In view of the definitions, it suffices to prove that if $k\in \Z$,
$t_k\in [-1,1]$, and $\wt u_k \in F_k$ then
\begin{equation}
\norm{\ft[\wt u_k(t_k)]}_{L_\xi^2}\les \norm{\ft[\wt u_k\cdot
\eta_0(2^{k_+}(t-t_k))]}_{B_k}.
\end{equation}
Let $f_k=\ft[\wt u_k\cdot \eta_0(2^{k_+}(t-t_k))]$, so
\[\ft[\wt u_k(t_k)](\xi)=c\int_\R f_k(\xi,\tau)e^{it_k\tau}d\tau.\]
From the definition of $B_k$, we get that
\[ \norm{\ft[\wt
u_k(t_k)]}_{L_\xi^2}\les \normo{\int_\R
|f_k(\xi,\tau)|d\tau}_{L_\xi^2}\les \norm{f_k}_{B_k},\] which
completes the proof of the proposition.
\end{proof}

\begin{proposition}\label{p92}
Assume $T\in (0,1]$, $u,v\in C([-T,T]:H^\infty)$ and
\begin{equation}\label{eq:lBO}
u_t+\Hl u_{xx}=v \mbox{ on } \R^2\times (-T,T).
\end{equation}
Then for any $l,s\geq 0$,
\begin{equation}\label{eq:p921}
\norm{u}_{F^{l,s}(T)}\les \
\norm{u}_{E^{l,s}(T)}+\norm{v}_{N^{l,s}(T)}.
\end{equation}
\end{proposition}
\begin{proof}
In view of the definitions, we see that the square of the right-hand
side of \eqref{eq:p921} is equivalent to
\begin{eqnarray}
&&\sum_{k\leq
0}\big(2^{2lk}\norm{R_k(u(0))}_{L^2}^2+2^{2lk}\norm{R_k(v)}_{N_k(T)}^2\big)\nonumber\\
&&+\sum_{k\geq 1}\big(\sup_{t_k\in
[-T,T]}2^{2sk}\norm{R_k(u(t_k))}_{L^2}^2
+2^{2sk}\norm{R_k(v)}_{N_k(T)}^2\big).
\end{eqnarray}
Thus, from definitions, it suffices to prove that if $k\in \Z$ and
$u,v \in C([-T,T]:H^\infty)$ solve \eqref{eq:lBO}, then
\begin{eqnarray}\label{eq:retardlinear}
\begin{array}{l}
\norm{R_k(u)}_{F_k(T)}\les
\norm{R_k(u(0))}_{L^2}+\norm{R_k(v)}_{N_k(T)} \mbox{ if } k\leq 0;\\
\norm{R_k(u)}_{F_k(T)}\les \sup_{t_k\in
[-T,T]}\norm{R_k(u(t_k))}_{L^2}+\norm{R_k(v)}_{N_k(T)} \mbox{ if }
k\geq 1.
\end{array}
\end{eqnarray}

Let $\wt v$ denote an extension of $R_k(v)$ such that $\norm{\wt
v}_{N_k}\leq C\norm{v}_{N_k(T)}$. Using \eqref{eq:Sk}, we may assume
that $\wt v$ is supported in $\R\times
[-T-2^{-k_+-10},T+2^{-k_+-10}]$, $k\in \Z$. Indeed, let $\beta(t)$
be a smooth function such that
\[\beta(t)=1, \mbox{ if }t\geq 1;\quad \beta(t)=0, \mbox{ if } t\leq 0.\]
Thus $\beta(2^{k_++10}(t+T+2^{-k_+-10}))$,
$\beta(-2^{k_++10}(t-T-2^{-k_+-10})) \in S_k$. Then we see that
$\beta(2^{k_++10}(t+T+2^{-k_+-10}))\beta(-2^{k_++10}(t-T-2^{-k_+-10}))
$ is supported in $[-T-2^{-k_+-10},T+2^{-k_+-10}]$, and equal to $1$
in $[-T,T]$. For $t\geq T$ we define
\[\wt u (t)=\eta_0(2^{k_++5}(t-T))[W(t-T)R_k(u(T))+\int_T^tW(t-s)(R_k(\wt v(s)))ds].\]
For $t\leq -T$ we define
\[\wt u (t)=\eta_0(2^{k_++5}(t+T))[W(t+T)R_k(u(-T))+\int_{-T}^tW(t-s)(R_k(\wt v(s)))ds].\]
For $t\in [-T,T]$ we define $\wt u(t)=u(t)$. It is clear that $\wt
u$ is an extension of u. Also, using \eqref{eq:Sk}, we get
\begin{eqnarray}\label{eq:extu}
\norm{u}_{F_k(T)}\les \sup_{t_k\in [-T,T]}\norm{\ft[\wt u \cdot
\eta_0(2^{k_+}(t-t_k))]}_{B_k}.
\end{eqnarray}
Indeed, to prove \eqref{eq:extu}, it suffices to prove that
\begin{eqnarray}
\sup_{t_k\in \R}\norm{\ft[\wt u \cdot
\eta_0(2^{k_+}(t-t_k))]}_{B_k}\les \sup_{t_k\in [-T,T]}\norm{\ft[\wt
u \cdot \eta_0(2^{k_+}(t-t_k))]}_{B_k}.
\end{eqnarray}
For $t_k>T$, since $\wt{u}$ is supported in
$[-T-2^{-k_+-5},T+2^{-k_+-5}]$, it is easy to see that
\[\wt{u}\eta_0(2^{k_+}(t-t_k))=\wt{u}\eta_0(2^{k_+}(t-T))\eta_0(2^{k_+}(t-t_k)).\]
Therefore, we get from \eqref{eq:pBk3} that
\[\sup_{t_k>T}\norm{\ft[\wt u \cdot
\eta_0(2^{k_+}(t-t_k))]}_{B_k}\les \sup_{t_k\in [-T,T]}\norm{\ft[\wt
u \cdot \eta_0(2^{k_+}(t-t_k))]}_{B_k}.\] Using the same way for
$t_k<-T$, we obtain \eqref{eq:extu} as desired.

It remains to prove \eqref{eq:retardlinear}. In view of the
definitions, \eqref{eq:extu} and \eqref{eq:pBk3}, it suffices to
prove that if $k\in \Z$, $\phi_k \in L^2$ with $\widehat{\phi_k}$
supported in $I_k$, and $v_k\in N_k$ then
\begin{eqnarray}
\norm{\ft[u_k\cdot \eta_0(2^{k_+}t)]}_{B_k}\les
\norm{\phi_k}_{L^2}+\norm{(\tau-\omega(\xi)+i2^{k_+})^{-1}\cdot
\ft(v_k)}_{B_k},
\end{eqnarray}
where
\begin{equation}
u_k(t)=W(t)(\phi_k)+\int_0^tW(t-s)(v_k(s))ds.
\end{equation}
Straightforward computations show that
\begin{eqnarray*}
&&\ft[u_k\cdot
\eta_0(2^{k_+}t)](\xi,\tau)=\widehat{\phi_k}(\xi)\cdot
2^{-k_+}\widehat{\eta_0}(2^{-k_+}(\tau-\omega(\xi)))\\
&& +C\int_\R \ft(v_k)(\xi,\tau')\cdot
\frac{2^{-k_+}\widehat{\eta_0}(2^{-k_+}(\tau-\tau'))-2^{-k_+}\widehat{\eta_0}(2^{-k_+}(\tau-\omega(\xi)))}{\tau'-\omega(\xi)}d\tau'.
\end{eqnarray*}
We observe now that
\begin{eqnarray*}
&&\aabs{\frac{2^{-k_+}\widehat{\eta_0}(2^{-k_+}(\tau-\tau'))-2^{-k_+}\widehat{\eta_0}(2^{-k_+}(\tau-\omega(\xi)))}{\tau'-\omega(\xi)}\cdot
(\tau'-\omega(\xi)+i2^{k_+})}\\
&&\les \
2^{-k_+}(1+2^{-k_+}|\tau-\tau'|)^{-4}+2^{-k_+}(1+2^{-k_+}|\tau-\omega(\xi)|)^{-4}.
\end{eqnarray*}
Using \eqref{eq:pBk1} and \eqref{eq:pBk2}, we complete the proof of
the proposition.
\end{proof}

We prove a crucial trilinear estimates in the following proposition.
\begin{proposition}\label{p93}
Let $0\leq l \leq 1/4$ and $s\geq 1/4$. Then
\begin{eqnarray}
\norm{\partial_x(uvw)}_{N^{l,s}(T)}&\les&
\norm{u}_{F^{l,s}(T)}\norm{v}_{F^{1/4,1/4}(T)}\norm{w}_{F^{1/4,1/4}(T)}\nonumber\\
&&+\norm{u}_{F^{1/4,1/4}(T)}\norm{v}_{F^{l,s}(T)}\norm{w}_{F^{1/4,1/4}(T)}\nonumber\\
&&+\norm{u}_{F^{1/4,1/4}(T)}\norm{v}_{F^{1/4,1/4}(T)}\norm{w}_{F^{l,s}(T)}.
\end{eqnarray}
\end{proposition}
\begin{proof}
Since $R_kR_j=0$ if $k\neq j$, then we can fix extensions $\wt{u},
\wt{v}, \wt{w}$ of $u, v, w$ such that $\norm{R_k(\wt u)}_{F_k}\leq
2\norm{R_k(u)}_{F_k(T)}$, $\norm{R_k(\wt v)}_{F_k}\leq
2\norm{R_k(v)}_{F_k(T)}$ and $\norm{R_k(\wt w)}_{F_k}\leq
2\norm{R_k(w)}_{F_k(T)}$ for any $k\in \Z$. In view of definition,
we get
\begin{eqnarray*}
\norm{\partial_x(\wt u \wt v \wt
w)}_{N^{l,s}}^2&=&\sum_{k_4=-\infty}^{-1}2^{2lk_4}\norm{R_{k_4}(\partial_x(\wt
u \wt v \wt
w))}_{N_{k_4}}^2\\
&&+\sum_{k_4=0}^{\infty}2^{2sk_4}\norm{R_{k_4}(\partial_x(\wt u \wt
v \wt w))}_{N_{k_4}}^2.
\end{eqnarray*}
Let $\wt{u}_{k}=R_{k}(\wt u)$, $\wt{v}_{k}=R_{k}(\wt v)$ and
$\wt{w}_{k}=R_{k}(\wt u)$. Then we get
\begin{eqnarray*}
\norm{R_{k_4}(\partial_x(\wt u \wt v \wt w))}_{N_{k_4}}\les
\sum_{k_1,k_2,k_3\in \Z}\norm{R_{k_4}(\partial_x(\wt u_{k_1} \wt
v_{k_2} \wt w_{k_3}))}_{N_{k_4}}.
\end{eqnarray*}
From symmetry we may assume $k_1\leq k_2\leq k_3$. Dividing the
summation into several parts, we get
\begin{eqnarray}\label{eq:shorttrilinear}
\sum_{k_1\leq k_2\leq k_3}\norm{R_{k_4}(\partial_x(\wt u_{k_1} \wt
v_{k_2} \wt w_{k_3}))}_{N_{k_4}}\leq \sum_{j=1}^6\sum_{\{k_i\}\in
A_j}\norm{R_{k_4}(\partial_x(\wt u_{k_1} \wt v_{k_2} \wt
w_{k_3}))}_{N_{k_4}},\quad
\end{eqnarray}
where $A_j$, $j=1,2,\ldots,6$, are as in the proof of Proposition
\ref{p63}. We will apply Proposition \ref{p71}-\ref{p76} obtained in
the Section 7 to bound the six terms in \eqref{eq:shorttrilinear}.
For example, for the first term, from Proposition 7.1, we have
\begin{eqnarray*}
&&\normb{2^{sk_4}\sum_{\{k_i\}\in A_1}\norm{R_{k_4}(\partial_x(\wt
u_{k_1} \wt
v_{k_2} \wt w_{k_3}))}_{N_{k_4}}}_{l_{k_4}^2}\\
&\leq& C\normb{2^{sk_4}\sum_{k_i\in A_1} \min(2^{k_1/2}, |k_2|+1)
\norm{\wt u_{k_1}}_{F_{k_1}}\norm{\wt v_{k_2}}_{F_{k_2}}\norm{\wt w_{k_3}}_{F_{k_3}}}_{l_{k_4}^2}\\
&\leq& \norm{\wt u}_{F^{1/4,1/4}}\norm{\wt v}_{F^{1/4,1/4}}\norm{\wt
w}_{F^{l,s}}.
\end{eqnarray*}
For the other terms we can handle them in the similar ways.
Therefore we complete the proof of the proposition.
\end{proof}

We prove now Theorem \ref{aprioribound}. Fix $0<l<1/4$ and $s>1/4$.
By the scaling \eqref{eq:scaling} we may assume that
\begin{eqnarray}\label{eq:normphi}
\norm{\phi}_{\dot{H}^l\cap \dot{H}^s}\leq \delta_0/M,
\end{eqnarray}
where $M\gg 1$ and $\delta_0$ is given as in Proposition
\ref{energyes}. For any $T'\in [0,1]$, we denote
$X(T')=\norm{u}_{E^{l,s}(T')}+\norm{\partial_x(u^3)}_{N^{l,s}(T')}$.
We assume first that $X(T')$ is continuous and satisfies
\begin{eqnarray}\label{eq:XTcont}
\lim_{T'\rightarrow 0}X(T')\les \norm{\phi}_{\dot{H}^l\cap
\dot{H}^s}.
\end{eqnarray}
We will prove that
\begin{eqnarray}\label{eq:XTbound}
X(T')\leq 2c\delta_0/M,\quad \mbox{ for any } T'\in [0,1].
\end{eqnarray}
 By using bootstrap (e.g. see \cite{Taolocal}), we may
assume that $X(T')\leq 3c\delta_0/M$ for any $T'\in [0,1]$. It
follows from Proposition \ref{p92}, \ref{p93} that for any $T'\in
[0,1]$ we have
\begin{eqnarray}\label{eq:semipert}
\left\{\begin{array}{l} \norm{u}_{F^{l,s}(T')}\les \norm{u}_{E^{l,s}(T')}+\norm{\partial_x(u^3)}_{N^{l,s}(T')};\\
\norm{\partial_x(u^3)}_{N^{l,s}(T')}\les \norm{u}_{F^{l,s}(T')}^3;\\
\norm{u}^2_{E^{l,s}(T')}\les \norm{\phi}_{\dot{H}^l\cap
\dot{H}^s}^2+\norm{u}_{F^{l,s}(T')}^6.
\end{array}
\right.
\end{eqnarray}
Thus we get
\begin{eqnarray}
X(T')^2\les  \norm{\phi}_{\dot{H}^l\cap \dot{H}^s}^2+X(T')^6,
\end{eqnarray}
from which and \eqref{eq:normphi} and the assumption $X(T')\leq
3c\delta_0/M$, we obtain \eqref{eq:XTbound} as desired. Then, using
\eqref{eq:semipert}, \eqref{eq:XTbound} and Proposition \ref{p91},
we have
\begin{eqnarray}
\norm{u}_{C([-1,1];\dot{H}^l\cap \dot{H}^s)}\les
\norm{\phi}_{\dot{H}^l\cap \dot{H}^s}.
\end{eqnarray}
For general $\phi$ and the $L^2$ norm of the solution, we just use
the scaling \eqref{eq:scaling} and the $L^2$ conservation law
\eqref{eq:L2con}.

It remains to prove that $X(T)$ is continuous and \eqref{eq:XTcont}.
Obviously, for $u\in C([-T,T]:H^\infty)$ the first component
$T'\rightarrow \norm{u}_{E^{l,s}(T')}$ is increasing and continuous
on $[-T,T]$ and
\[\lim_{T'\rightarrow 0}\norm{u}_{E^{l,s}(T')}\les \norm{\phi}_{\dot{H}^l\cap
\dot{H}^s}.\] For the second component it follows from the similar
argument as in the proof of Lemma 4.2 in \cite{IKT}. We omit the
details.

\begin{remark}
From the proof we see that we actually prove a stronger result than
that stated in Theorem \ref{aprioribound}. We expect that some
wellposedness results hold for the initial data in $\dot{H}^l\cap
\dot{H}^s$.
\end{remark}

\noindent{\bf Acknowledgment.} This work was finished under the
patient advising of Prof. Carlos E. Kenig while the author was
visiting the Department of Mathematics at the University of Chicago
under the auspices of China Scholarship Council. The author is
deeply indebted to Prof. Kenig for the many encouragements and
precious advices. The author is also very grateful to Prof. Heping
Liu and Prof. Jianmin Liu, and especially to his advisors Prof.
Lizhong Peng and Prof. Baoxiang Wang for their encouragements,
supports and advisings.

\end{document}